\documentclass[a4paper,11pt]{article}
\usepackage[utf8]{inputenc}
\usepackage{enumerate}
\usepackage{lmodern}
\usepackage[T1]{fontenc}
\usepackage{verbatim}
\usepackage{textcomp}
\usepackage[english]{babel}
\usepackage[a4paper,vmargin={3.5cm,3.5cm},hmargin={2.5cm,2.5cm}]{geometry}
\usepackage[font=sf, labelfont={sf,bf}, margin=1cm]{caption}
\usepackage[pdftex]{hyperref}

\usepackage[pdftex]{color,graphicx}

\usepackage{amsmath,amsfonts,amssymb,amsthm,mathrsfs}

\newtheorem{theorem}{Theorem}[section]
\newtheorem{definition}[theorem]{Definition}
\newtheorem{lemma}[theorem]{Lemma}

\newtheorem{corollary}[theorem]{Corollary}
\newtheorem{proposition}[theorem]{Proposition}

\def\llbracket{[\hspace{-.10em} [ }
\def\rrbracket{ ] \hspace{-.10em}]}
\def\pp{{\mathcal P}}
\def\t{{\mathcal T}}
\def\z{{\mathcal Z}}
\def\n{{\mathcal N}}
\def\w{{\mathrm w}}
\def\g{{\mathcal G}}
\def\y{{\mathcal Y}}
\def\x{{\mathcal X}}
\def\r{{\mathcal R}}
\def\b{{\mathcal B}}

\def\ve{\varepsilon}
\def\wt{\widetilde}
\def\wh{\widehat}
\def\bm{{\bf m}}

\def\la{\longrightarrow}
\def\T{{\mathbb T}}
\def\K{{\mathbb K}}
\def\N{{\mathbb N}}
\def\E{{\mathbb E}}
\def\P{{\mathbb P}}
\def\R{{\mathbb R}}

\def\ind{\mathbf{1}_}

\def\rd{\mathrm{d}}
\def\build#1_#2^#3{\mathrel{
\mathop{\kern 0pt#1}\limits_{#2}^{#3}}}
\def\rem{\noindent{\bf Remark. }}

\title{The hull process of the Brownian plane}
\author{Nicolas Curien and Jean-Fran\c cois Le Gall}

\date{\small \it Universit\'e Paris-Sud}
\begin{document}
\maketitle

\begin{abstract}
We study the random metric space called the Brownian plane, which is closely related to the Brownian map
and is conjectured to be the 
universal scaling limit of many discrete random lattices such as the uniform infinite planar triangulation.
We obtain a number of explicit distributions for the Brownian plane. In particular, we consider, for every
$r>0$, the hull of radius $r$, which is obtained by ``filling in the holes'' in the ball of radius $r$ centered
at the root. We 
introduce a quantity $Z_r$ which is interpreted as the (generalized) length of the boundary of the 
hull of radius $r$. We identify the law of the process $(Z_r)_{r>0}$ as the  time-reversal of a 
continuous-state branching process starting from $+\infty$ at time $-\infty$ and conditioned to
hit $0$ at time $0$, and we give an explicit description of the process of hull volumes
given the process $(Z_r)_{r>0}$. We obtain an explicit formula for the Laplace transform of the volume of the hull of radius $r$,
and we also determine the conditional distribution of this volume given the 
length of the boundary. Our proofs involve certain new formulas for
super-Brownian motion and the Brownian snake in dimension one, which are
of independent interest.
\end{abstract}

\section{Introduction}

Much recent work has been devoted to understanding continuous limits of random graphs
drawn on the two-dimensonal sphere or in the plane, which are called random planar maps. 
A fundamental object is the random compact metric space known as the Brownian map, which has been proved to be the
universal scaling limit of several important classes of random planar maps conditioned to
have a large size (see in particular \cite{Ab,AA,BJM,LGU,Mie}). 
The main goal of this work is to study the random (non-compact) metric space called the Brownian plane, which 
may be viewed as
an infinite-volume version of the Brownian map. The Brownian plane was first introduced and
studied in \cite{CLG}, where it was shown to be the scaling limit in distribution of the uniform infinite planar quadrangulation (UIPQ) in the 
local Gromov-Hausdorff sense. The Brownian plane is in fact conjectured
to be the universal scaling limit of  many discrete random lattices
including the uniform infinite planar triangulation (UIPT) introduced by Angel and Schramm
\cite{AS} and studied then by several authors. It was proved in \cite{CLG} that the Brownian plane is locally isometric 
to the Brownian map, in the following sense. Recalling that both the Brownian map and the Brownian plane
are equipped with a distinguished point called the root, 
one can couple these two random metric spaces in such a way that, for every $\delta>0$, 
there exists $\ve>0$ such that the balls of radius $\ve$ centered at the root in the two spaces are isometric with probability
at least $1-\delta$. As a consequence, the Brownian plane shares many properties of the Brownian map.
On the other hand, the Brownian plane also
enjoys the important additional property of invariance under scaling: Multiplying the distance 
by a constant factor $\lambda >0$ does not change the distribution of the Brownian plane. This property
suggests that the Brownian plane should be more tractable for calculations than the Brownian map, for which
very few explicit distributions are known. Our purpose is to obtain such explicit distributions for the Brownian plane, and
in particular to give a detailed probabilistic description of the growth of ``hulls'' centered at the root.

In order to give a more precise presentation of our results, let us introduce some notation. As in \cite{CLG}, we write $(\pp_\infty, D_\infty)$
for the Brownian plane, and we let $\rho_\infty$ stand for the distinguished point of $\pp_\infty$ called the root. We recall that
$\pp_\infty$ is equipped with a volume measure, and we write $|A|$ for the volume of a measurable subset 
of $\pp_\infty$. For every
$r> 0$, the closed ball of radius $r$ centered at $\rho_\infty$ in $\pp_\infty$ is denoted by $B_r(\pp_\infty)$. In contrast with
the case of Euclidean space, the complement of $B_r(\pp_\infty)$ will have infinitely many connected components
(see \cite{LGcactus} for a detailed discussion of these components in the slightly different setting of the Brownian map)
but only one unbounded connected component.
We then define
the hull of radius $r$
 as the complement of the unbounded component of the complement of $B_r(\pp_\infty)$, and
 we denote this hull by $B^\bullet_r(\pp_\infty)$. Informally, $B^\bullet_r(\pp_\infty)$ is obtained by
  ``filling in the holes'' of $B_r(\pp_\infty)$ - see Fig.~1 below, and Fig.~3 in Section 5  for a discrete version of the hull.
  
 In what follows, we give a complete description of the
  law of the process $(|B^\bullet_r(\pp_\infty)|)_{r>0}
 $. To formulate this description, it is convenient to introduce another process $(Z_r)_{r>0}$ which gives for
 every $r>0$ the size of the boundary of $B^\bullet_r(\pp_\infty)$.
 
\begin{proposition}
\label{approx-exit}
Let $r>0$. There exists a positive random variable $Z_r$ such that
$$\lim_{\ve\to 0} \ve^{-2}|B_r^\bullet(\mathcal{P}_\infty)^c\cap B_{r+\ve}(\mathcal{P}_\infty)|= Z_r$$
in probability.
\end{proposition}

\begin{figure}[!h]
\label{d-hull}
 \begin{center}
 \includegraphics[width=10cm]{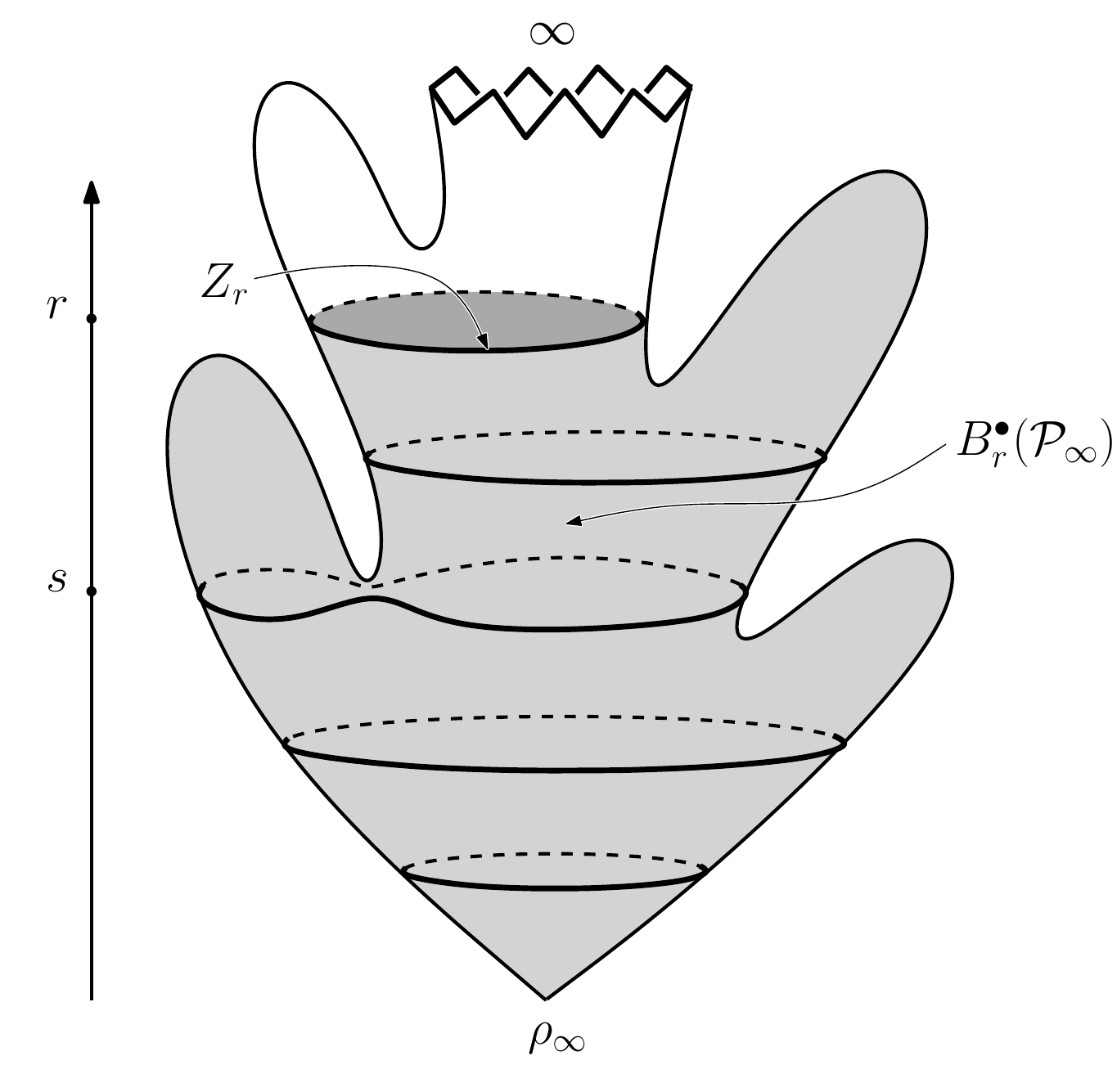}
 \caption{Illustration of the geometric meaning of the processes $(Z_r)_{r\geq
0}$ and $(|B_r^\bullet(\mathcal{P}_\infty)|)_{r\geq 0}$. The Brownian
plane is represented as a two-dimensional ``cactus'' where the height
of each point is equal to its distance to the root. The shaded part represents the 
hull $B_r^\bullet(\mathcal{P}_\infty)$. At time $s$, both
processes $Z_\cdot$ and $ |B_{\cdot}^\bullet( \mathcal{P}_\infty)|$
have a jump.  Geometrically this corresponds to the creation of a
``bubble'' above height $s$.}
 \end{center}
 \vspace{-2mm}
 \end{figure}

  In view of this proposition, one interprets $Z_r$ as the (generalized) length of the boundary of the 
hull of radius $r$ (this boundary is expected to be a fractal curve of dimension $2$). A key intermediate step in the derivation of our main results is to identify 
the process $(Z_r)_{r>0}$ as a time-reversed continuous-state branching process. 
For every $u\geq 0$, set $\psi(u)=\sqrt{8/3}\,u^{3/2}$. 
The 
continuous-state branching process with branching mechanism $\psi$ is
the Feller Markov process $(X_t)_{t\geq 0}$ with values in $\R_+$, whose semigroup is characterized as
follows: for
every $x,t\geq 0$ and every $\lambda> 0$,
$$E[e^{-\lambda X_t}\mid X_0=x]= \exp\Big(-x\Big(\lambda^{-1/2} + \sqrt{2/3}\;t\Big)^{-2}\Big).$$
See subsection \ref{CSBPpreli} for a brief discussion of this process. 
Note that $X$ gets absorbed at $0$ in finite time. It is easy to construct a process
$(\wt X_t)_{t\leq 0}$ indexed by the time interval $(-\infty,0]$ and which is distributed
as the process $X$ ``started from $+\infty$'' at time $-\infty$ and conditioned to hit zero at time $0$
(see subsection \ref{CSBPpreli} for a more rigorous presentation). 

\begin{proposition}
\label{process-exit}
{\rm (i)} For every $r>0$, we have for every $\lambda\geq 0$,
$$E\Big[ \exp(-\lambda Z_r)]= \Big(1+\frac{2\lambda r^2}{3}\Big)^{-3/2}.$$
Equivalently, $Z_r$ follows a Gamma distribution with parameter $\frac{3}{2}$ and mean $r^2$.

\noindent{\rm (ii)} The two processes $(Z_r)_{r>0}$ and $(\wt X_{-r})_{r>0}$ have the same finite-dimensional marginals.

\end{proposition}

We observe that  results closely related to Proposition \ref{process-exit} have been obtained by Krikun \cite{Kr1,Kr2} in the discrete setting of the
UIPT and the UIPQ. 

Part (ii) of the preceding proposition implies that the process $(Z_r)_{r>0}$ has a c\`adl\`ag modification, with only
negative jumps, and from now on we deal with this modification. We can now state the
main results of the present work. For every $r>0$, we write $\Delta Z_r$ for the jump of $Z$ at time $r$.

\begin{theorem}
\label{hull-process-description}
Let $s_1,s_2,\ldots$ be a measurable enumeration of the jumps of $Z$, and let
$\xi_1,\xi_2,\ldots$ be a sequence of i.i.d. real random variables with density
$$\frac{1}{\sqrt{2\pi x^5}}\,e^{-1/2x}\,\ind{(0,\infty)}(x),$$
which is independent of the process $(Z_r)_{r> 0}$.
The following identity in distribution of random processes holds:
$$\Big( Z_r, |B^\bullet_r(\pp_\infty)|\Big)_{r>0}
\build{=}_{}^{\rm (d)} \Big(Z_r,\sum_{i:s_i\leq r} \xi_i\,(\Delta Z_{s_i})^2\Big)_{r>0}.$$
\end{theorem}
  
This theorem identifies the conditional distribution of the process of hull volumes knowing
the process of hull boundary lengths, whose distribution is given by the preceding proposition.
Informally, each jump time $r$ of $Z$ corresponds to the creation of a new connected component of
the complement of the ball $B_r(\pp_\infty)$, which is ``swallowed'' by the hull, leading to
a negative jump for the boundary of the hull and a positive jump for its volume. The common distribution
of the variables $\xi_i$ should then be interpreted as the law of the volume of a
newly created connected component knowing that the ``length'' of its boundary is equal to $1$
(see \cite[Proposition 6.4]{Ang} for a related result
concerning the asymptotic distribution of the volume of a triangulation with a boundary
of size tending to infinity). This 
heuristic discussion is made much more precise in the companion paper \cite{CLG2}, where
many of the results of the present work are interpreted in terms of asymptotics for the so-called 
``peeling process'' studied by Angel \cite{Ang} for the UIPT. 

The proof of Theorem \ref{hull-process-description} depends on certain explicit calculations of
distributions, which are of independent interest.

\begin{theorem}
\label{laws-hull}
Let $r>0$. For every $\mu>0$,
$$E\Big[\exp(-\mu|B^\bullet_r(\pp_\infty)|)\Big] = 3^{3/2}\,\cosh((2\mu)^{1/4}r)
\,\Big(\cosh^2((2\mu)^{1/4}r) + 2\Big)^{-3/2}.$$
Furthermore, for every $\ell>0$,
\begin{align*}
&E\Big[\exp(-\mu|B^\bullet_r(\pp_\infty)|)\,\Big|\, Z_r=\ell\Big] \\
&\qquad= r^3(2\mu)^{3/4} \,\frac{\cosh((2\mu)^{1/4}r)}{\sinh^3((2\mu)^{1/4}r)}
\,\exp\Big(-\ell\Big(\sqrt{\frac{\mu}{2}}\Big(3 \coth^2((2\mu)^{1/4}r) -2\Big)-\frac{3}{2r^2}\Big)\Big).
\end{align*}
\end{theorem}

In view of the first assertion of the theorem, one may ask whether a similar formula 
holds for the volume $|B_r(\pp_\infty)|$ of the ball of radius $r$. In principle our
methods should also be applicable to this problem, but our calculations did not lead to
a tractable expression. One may still compare the expected volumes of the hull and the
ball. From the first formula of the theorem, one easily gets that $E[|B^\bullet_r(\pp_\infty)|]=r^4/3$.
On the other hand, using the method of the proof of \cite[Proposition 5]{LGM}, one can verify that
$E[|B_r(\pp_\infty)|]=2r^4/21$.

We also note that 
there is an interesting analogy between the second formula of Theorem \ref{laws-hull} and
classical formulas for Bessel processes (see Corollary 1.8 and Corollary 3.3 in
\cite[Chapter XI]{RY}), which also involve hyperbolic functions -- in special cases
these formulas can be restated in terms of linear Brownian motion via the Ray-Knight theorems.

The preceding results can also be interpreted in terms of asymptotics for the UIPQ.
In the last section of this article, we prove that the process of hull volumes of the UIPQ
converges in distribution, modulo a suitable rescaling, to the process $(|B^\bullet_r(\pp_\infty)|)_{r>0}$.
A similar invariance principle should hold for the UIPT and for more general random lattices such as
the ones constructed by Addario-Berry \cite{Add} and Stephenson \cite{Ste}. 

Our proofs depend on a new representation of the Brownian plane, which is different from
the one used in \cite{CLG}. Roughly speaking, this representation is a continuous analog
of the construction of the UIPQ that was given by Chassaing and Durhuus in \cite{CD},
whereas \cite{CLG} used a continuous version of the construction in \cite{CMM}.  
Similarly as in \cite{CLG}, the representation of the Brownian plane in the present work uses a random infinite 
real tree $\t_\infty$ whose vertices are assigned real labels. The probabilistic structure of the real tree $\t_\infty$
is more complicated than in \cite{CLG}, but the labels are now nonnegative and correspond to distances
from the root in $\pp_\infty$ (whereas in \cite{CLG} labels corresponded in some sense to ``distances
from infinity''). This is of course similar to the well-known Schaeffer bijection between rooted
quadrangulations and well-labeled trees \cite{CS}. The fact that labels are distances from the root 
is important for our purposes, since it allows us to give a simple representation of the hull 
of radius $r$: The complement of this hull corresponds to the set of all points $a$ in
$\t_\infty$ such that labels stay greater than $r$ along the (tree) geodesic from $a$ to infinity. 
See formula \eqref{formula-hull} below.
There is a similar interpretation for the boundary of the hull, and a key observation is the fact that the ``boundary length'' $Z_r$ can be obtained in terms
of exit measures from $(r,\infty)$ associated with the ``subtrees'' branching off the spine of the infinite tree $\t_\infty$
at a level greater than the last occurence of label $r$ on the spine (see formula \eqref{size-boundary} below). 

The construction of the infinite tree $\t_\infty$ and of the labels assigned to
its vertices, as well as the subsequent calculations, make
a heavy use of the Brownian snake and its properties. In particular the special Markov
property of the Brownian snake \cite{LG0} and its connections with partial differential 
equations play an important role. Because of the close relation between super-Brownian 
motion and the Brownian snake, some of the results that follow can be written as statements
about super-Brownian motion, which may be of independent interest. In particular, Corollary 
\ref{supercoro}, which is essentially equivalent to the second formula of  Theorem \ref{laws-hull}, gives the 
Laplace transform of the total integrated mass of a super-Brownian motion started from $u\delta_a$
(for some $u,a>0$)
knowing that the minimum of the range is equal to $0$. Similarly, Corollary
\ref{lawprocess} determines for a super-Brownian motion starting from $\delta_0$ the law of the process 
whose value at time $r>0$ is the pair consisting of the exit measure from $(-r,\infty)$ and the mass of
those historical paths that do not hit level $-r$. 

The paper is organized as follows. Section 2 presents a number of preliminaries. In particular,
we recall basic facts about the (one-dimensional) Brownian snake including  
exit measures and the special Markov property, and its connections
with super-Brownian motion. We also state a recent result from \cite{Bessel} giving a decomposition
of the Brownian snake knowing its minimal spatial position. The latter result is especially useful in Section 3, 
where we derive our new representation of the 
Brownian plane. In order to show that this new construction is equivalent to the one in \cite{CLG}, we use the
fact that the distribution of the  Brownian plane is characterized by the invariance under scaling and the above-mentioned
property stating that the Brownian plane is locally isometric to the Brownian map. Section 4 contains the proof
of our main results: Propositions \ref{approx-exit} and \ref{process-exit} are proved in subsection \ref{exit-process},
Theorem \ref{laws-hull} is derived in subsection \ref{volu-hull}, and Theorem \ref{hull-process-description}
is proved in subsection \ref{pro-hull}.
 Finally, Section 5 is devoted to our invariance principle relating the hull process 
of the UIPQ to the process $(|B^\bullet_r(\pp_\infty)|)_{r>0}$.

\section{Preliminaries}

\subsection{A continuous-state branching process}
\label{CSBPpreli}

An important role in this work will be played by 
a particular continuous-state branching process, which was
already mentioned in the introduction. We refer to \cite[Chapter 2]{LGZ}
and references therein for the general theory of continuous-state branching processes, and
content ourselves with a brief exposition of the case of interest in this work.
We fix a constant $c>0$. The 
continuous-state branching process with branching mechanism $\psi(u)=c\,u^{3/2}$ is
the Feller Markov process $(X_t)_{t\geq 0}$ with values in $\R_+$, c\`adl\`ag paths and no
negative jumps, whose semigroup is characterized as
follows. If $P_x$ stands for the probability measure under which $X$ starts from $X_0=x$, then, for
every $x,t\geq 0$ and every $\lambda> 0$,
$$E_x[e^{-\lambda X_t}]= e^{-x\,u_t(\lambda)}$$
where the function $u_t(\lambda)$ is determined by the differential equation
$$\frac{\rd u_t(\lambda)}{\rd t}= - c(u_t(\lambda))^{3/2}\;,\quad u_0(\lambda)=\lambda.$$
It follows that
$u_t(\lambda)=(\lambda^{-1/2} + \frac{c}{2}t)^{-2}$,
and thus,
\begin{equation}
\label{LaplaceCSBP}
E_x[e^{-\lambda X_t}]=\exp\Big(-x\Big(\lambda^{-1/2} + \frac{c}{2}t\Big)^{-2}\Big).
\end{equation}
By differentiating with respect to $\lambda$, we have also
\begin{equation}
\label{CSBPtech}
E_x[X_te^{-\lambda X_t}] = x\lambda^{-3/2}\,\Big(\lambda^{-1/2}+\frac{c}{2}t\Big)^{-3}\,\exp\Big(-x\Big(\lambda^{-1/2} + \frac{c}{2}t\Big)^{-2}\Big).
\end{equation}
Let $T:=\inf\{t\geq 0:X_t=0\}$, and note that $X_t=0$ for every $t\geq T$, a.s. Since $P_x(T\leq t)=P_x(X_t=0)=\exp(-\frac{4x}{c^2t^2})$, 
we readily obtain that the density of $T$ under $P_x$ is (when $x>0$) the function 
$$t\mapsto \phi_t(x):=\frac{8x}{c^2t^3}\,\exp\Big(-\frac{4x}{c^2t^2}\Big).$$

For future purposes, it will be useful to introduce the process $X$ conditioned on 
extinction at a fixed time.
To this end, we write $q_t(x,\rd y)$ for the transition kernels of $X$. We fix $\rho>0$ and define the process $X$ 
``conditioned on extinction at time $\rho$'' as the time-inhomogeneous Markov process indexed by the interval $[0,\rho]$
with values in $(0,\infty)$ (with $0$ serving as a cemetery point) whose transition 
kernel between times $s$ and $t$ is
$$\pi_{s,t}(x,\rd y)= \frac{\phi_{\rho-t}(y)}{\phi_{\rho-s}(x)}\,q_{t-s}(x,\rd y),$$
if $0\leq s<t<\rho$ and $x>0$, and
$$\pi_{s,\rho}(x,\rd y)=\delta_0(\rd y)$$
if $s\in[0,\rho)$ and $x>0$. This is just a standard $h$-transform in a time-inhomogeneous setting, and the interpretation
can be justified by the fact that, for every choice of $0<s_1<\cdots<s_p<\rho$, the conditional distribution 
of $(X_{s_1},\ldots,X_{s_p})$ under $P_x(\cdot\mid \rho\leq T<\rho+\ve)$ converges to $\pi_{0,s_1}(x,\rd y_1)\pi_{s_1,s_2}(y_1,\rd y_2)\ldots
\pi_{s_{p-1},s_p}(y_{p-1},\rd y_p)$ as $\ve\downarrow 0$. 

If $0\leq s<t<\rho$, the Laplace transform of $\pi_{s,t}(x,\rd y)$ is
\begin{align}
\label{CSBPtech2}
\int e^{-\lambda y}\,\pi_{s,t}(x,\rd y)&=  \frac{1}{\phi_{\rho-s}(x)}\,E_x[\phi_{\rho-t}(X_{t-s})\,e^{-\lambda X_{t-s}}]\nonumber\\
&=\Bigg(\frac{\rho-s}{\rho-t+ (t-s)(1+\frac{c^2}{4}\lambda(\rho-t)^2)^{1/2}}\Bigg)^3\nonumber\\
&\ \times \exp\Bigg(-\frac{4x}{c^2}\Big(\Big((\frac{c^2\lambda}{4} +(\rho-t)^{-2})^{-1/2} +t-s\Big)^{-2}- (\rho-s)^{-2}\Big)\Bigg),
\end{align}
where the second equality follows from the explicit expression of $\phi_{\rho-s}$ and formula \eqref{CSBPtech}.

Finally, let us briefly discuss the process $\wt X$ which was introduced in Section 1. Simple arguments give the existence of a 
process $(\wt X_t)_{t\in(-\infty,0]}$ with c\`adl\`ag paths and no negative jumps, which is
indexed by the time interval $(-\infty,0]$ and such that:
\begin{enumerate}
\item[$\bullet$] $\wt X_t>0$ for every $t<0$, and $\wt X_0=0$, a.s.;
\item[$\bullet$] $\wt X_t\la +\infty$ as $t\downarrow -\infty$, a.s.;
\item[$\bullet$] for every $x>0$, if $\wt T_x:=\inf\{t\in(-\infty,0]: \wt X_t\leq x\}$, the process
$(\wt X_{(\tilde T_x+t)\wedge 0})_{t\geq 0}$ has the same distribution as $X$ started from $x$. 
\end{enumerate}
To get an explicit construction of $\wt X$, one may concatenate independent copies of the process $X$ started
at $n$ and stopped at the hitting time of $n-1$, for every integer $n\geq 1$. We omit the details.

\subsection{Preliminaries about the Brownian snake}
\label{prelisnake}

We give below a brief presentation of the Brownian snake, referring to the book 
\cite{LGZ} for more details. We write $\mathcal{W}$ for the set of all finite paths in $\R$. An element of $\mathcal{W}$ is
a continuous mapping $\w:[0,\zeta]\longrightarrow \R$, where $\zeta=\zeta_{(\w)}\geq 0$ depends
on $\w$ and is called the lifetime of $\w$. We write $\wh \w=\w(\zeta_{(\w)})$ for the
endpoint of $\w$. For $x\in\R$, we set $\mathcal{W}_x:=\{\w\in \mathcal{W}:\w(0)=x\}$. 
The trivial path $\w$ such that $\w(0)=x$ and $\zeta_{(\w)}=0$
is identified with the point $x$ of $\R$, so that we can view $\R$ as a subset of $\mathcal{W}$.
The space $\mathcal{W}$ is equipped with the distance
$$d(\w,\w')=|\zeta_{(\w)}-\zeta_{(\w')}| + \sup_{t\geq 0} |(\w(t\wedge \zeta_{(\w)})-\w'(t\wedge \zeta_{(\w')})|.$$

The Brownian snake $(W_s)_{s\geq 0}$ is a continuous Markov process with values in $\mathcal{W}$. We will write 
$\zeta_s=\zeta_{(W_s)}$ for the lifetime process of $W_s$. The process $(\zeta_s)_{s\geq 0}$ evolves like a
reflecting Brownian motion in $\R_+$. Conditionally on $(\zeta_s)_{s\geq 0}$, the evolution of $(W_s)_{s\geq 0}$
can be described informally as follows: When $\zeta_s$ decreases, the path $W_s$ is shortened from its tip,
and when $\zeta_s$ increases the path $W_s$ is extended by adding ``little pieces of linear Brownian motion''
at its tip. We refer to \cite[Chapter IV]{LGZ} for a more rigorous presentation.

It is convenient to assume that the Brownian snake is defined on the canonical space $C(\R_+,\mathcal{W})$ of all continuous functions from $\R_+$ into $C(\R_+,\mathcal{W})$,
in such a way that, for $\omega=(\omega_s)_{s\geq 0}\in C(\R_+,\mathcal{W})$, we have $W_s(\omega)=\omega_s$. 
The notation $\P_\w$ then stands for the law of the Brownian snake started from $\w$. 

For every $x\in\R$, the trivial path $x$ is a regular recurrent point for the Brownian snake, and so we can make 
sense of the excursion measure $\N_x$ away from $x$, which is a $\sigma$-finite measure on $C(\R_+,\mathcal{W})$.
Under $\N_x$, the process $(\zeta_s)_{s\geq 0}$ is distributed according to the It\^o measure of positive excursions
of linear Brownian motion, which is normalized so that, for every $\ve>0$, 
$$\N_x\Big(\sup_{s\geq 0} \zeta_s >\ve\Big) =\frac{1}{2\ve}.$$
We write $\sigma:=\sup\{s\geq 0: \zeta_s>0\}$ for the duration of the excursion under $\N_x$. 
For every $\ell>0$, we will also use the notation $\N_0^{(\ell)}:=\N_0(\cdot\mid \sigma=\ell)$. 

We set
$$\r:=\{\wh W_s:s\geq 0\}\;,\ W_*:=\inf \r=\inf_{s\geq 0} \wh W_s.$$
We will consider $\r$ and $W_*$ under the excursion measures $\N_x$, and we note that
we have also $\r=\{\wh W_s:0\leq s\leq \sigma\}$ and $W_*=\min\{\wh W_s:0\leq s\leq \sigma\}$, $\N_x$ a.e.
Occasionally we also write $\omega_*=W_*(\omega)$ for $\omega\in C(\R_+,\mathcal{W})$.

If $x,y\in\R$ and $y<x$, we have
\begin{equation}
\label{hittingzero}
\N_x(y\in\r)=\N_x(W_*\leq y)= \frac{3}{2(x-y)^2}
\end{equation}
(see e.g. \cite[Section VI.1]{LGZ}). 

It is known (see e.g.~\cite[Proposition 2.5]{LGW}) that $\N_x$ a.e. there is a unique instant $s_{\bm}\in[0,\sigma]$ such that
$\wh W_{s_{\bm}}=W_*$. 

\medskip
\noindent{\bf Decomposing the Brownian snake at its minimum.} We will now recall a
key result of \cite{Bessel} that plays an important role in what follows. This result identifies the 
law of the minimizing path $W_{s_\bm}$ under $\N_0$, together with the distribution of the
``subtrees'' that branch off the minimizing path. Let us define these subtrees in a more
precise way.

For every $s\geq 0$, we set
$$\hat \zeta_s:=\zeta_{(s_{\bm}+s)\wedge \sigma}\;,\quad\check \zeta_s:=\zeta_{(s_{\bm}-s)\vee0}\;.$$
We let $(\hat a_i,\hat b_i)$, $i\in \hat I$ be the excursion intervals of $\hat \zeta_s$ above its past minimum. Equivalently,
the intervals $(\hat a_i,\hat b_i)$, $i\in \hat I$ are the connected components of 
the set 
$$\Big\{s\geq 0: \hat \zeta_s > \min_{0\leq r\leq s}\hat\zeta_r\Big\}.$$
Similarly, we let $(\check a_j,\check b_j)$, $j\in \check I$ be the excursion intervals of $\check\zeta_s$ above its past minimimum.
We may assume that the indexing sets $\hat I$ and $\check I$ are disjoint. In terms of the tree $\t_\zeta$ coded by
the excursion $(\zeta_s)_{0\leq s\leq \sigma}$ under $\N_0$ (see e.g. \cite[Section 2]{LGtree}),
each interval $(\hat a_i,\hat b_i)$ or $(\check a_j,\check b_j)$ corresponds to a subtree of $\t_\zeta$ branching off
the ancestral line of the vertex associated with $s_{\bm}$.
We next consider the spatial displacements corresponding to these subtrees. 
The properties of the Brownian snake imply that, for every $i\in I$, the paths $W_{s_\bm+s}$, $s\in [\hat a_i,\hat b_i]$,
are the same up to time $\zeta_{s_\bm+\hat a_i}=\zeta_{s_\bm+\hat b_i}$, and similarly for the paths
$W_{s_\bm-s}$, $s\in[\check a_j,\check b_j]$, for every $j\in J$. 
Then, for every $i\in \hat I$, we let $W^{[i]}\in C(\R_+,\mathcal{W})$ be defined by
$$W^{[i]}_s(t) = W_{s_{\bm}+(\hat a_i+s)\wedge\hat b_i} (\zeta_{s_{\bm}+\hat a_i} +t)\ , \quad 0\leq t\leq \zeta_{s_{\bm}+(\hat a_i+s)\wedge\hat b_i} -\zeta_{s_{\bm}+\hat a_i}.$$
Similarly, for every $j\in \check I$,
$$W^{[j]}_s(t) = W_{s_{\bm}-(\check a_j+s)\wedge \check b_j} (\zeta_{s_{\bm}-\check a_j} +t)\ , \quad 0\leq t\leq \zeta_{s_{\bm}-(\check a_j+s)\wedge\check b_j} -\zeta_{s_{\bm}-\check a_j}.$$
We finally introduce the point measures on $\R_+\times C(\R_+,\mathcal{W})$ defined by
$$\hat{\mathcal N} = \sum_{i\in \hat I} \delta_{(\zeta_{s_{\bm}+\hat a_i}, W^{[i]})}\ ,\quad \check{\mathcal N} = \sum_{j\in \check I} \delta_{(\zeta_{s_{\bm}-\hat a_j}, W^{[j]})}.$$

\begin{theorem}
\label{decotheo}
{\rm (i)}
Let $a>0$. Under the excursion measure $\N_0$ and
conditionally on $W_*=-a$, the random path $(a+W_{s_{\bm}}(\zeta_{s_{\bm}}-t))_{0\leq t\leq \zeta_{s_{\bm}}}$ is distributed as
a nine-dimensional Bessel process started from $0$ and stopped at its last passage time at level $a$.

{\rm (ii)} Under $\N_0$, conditionally on the minimizing path $W_{s_{\bm}}$, the point measures $\hat{\mathcal N}(\rd t,\rd\omega)$ and $ \check{\mathcal N}(\rd t,\rd\omega)$
are independent and their common conditional distribution is that of a Poisson point measure with intensity
$$2\,\mathbf{1}_{[0,\zeta_{s_{\bm}}]}(t)\,\mathbf{1}_{\{\omega_*> \wh W_{s_\bm}\}}\,\rd t\,\N_{W_{s_{\bm}}(t)}(\rd \omega).$$
\end{theorem}

We refer to \cite[Chapter XI]{RY} for basic facts about Bessel processes.
Parts (i) and (ii) of the theorem correspond respectively to Theorem 5 and Theorem 6 of \cite{Bessel}. 
Note that when applying Theorem 5 of \cite{Bessel}, we also use the fact that the time-reversal 
of a Bessel process of dimension $-5$ started from $a$ and stopped when hitting $0$ is a 
nine-dimensional Bessel process started from $0$ and stopped at its last passage time at level $a$
(see e.g. \cite[Exercise XI.1.23]{RY}).

\medskip
\noindent{\bf Exit measures and the special Markov property.} Let $D$ be an open interval of $\R$, such that $D\not =\R$. We fix
$x\in D$ and, for every $\w\in \mathcal{W}_x$, set
$$\tau_D(\w)=\inf\{t\in [0,\zeta_{(\w)}]:\w(t)\notin D\},$$ 
with the usual convention $\inf\varnothing=\infty$.
The exit measure $\z^D$ from $D$ (see \cite[Chapter 5]{LGZ}) is a random measure on $\partial D$, 
which is defined under $\N_x$ and is supported on the set of all exit points $W_s(\tau_D(W_s))$  for the paths $W_s$ such that  $\tau_D(W_s)<\infty$ (note that here $\partial D$ has at most two points, but the
preceding discussion remains valid for the $d$-dimensional Brownian snake
and an arbitrary subdomain $D$ of $\R^d$). Note that $\N_x(\z^D\not =0)<\infty$. It is easy to prove, for instance
by using Proposition \ref{SMP} below, that
\begin{equation}
\label{range-exit}
\{\z^D=0\}= \{\mathcal{R}\subset D\},\quad\N_x\ \hbox{a.e.}
\end{equation}

A crucial ingredient of our study is the special Markov property of the Brownian snake
\cite{LG0}.
In order to state this property, we first observe that, $\N_x$-a.e., the set
$$\{s\geq 0: \tau_D(W_s)<\zeta_s\}$$
is open and thus can be written as a union of disjoint open intervals
$(a_i,b_i)$, $i\in I$, where $I$ may be empty. From the properties of the Brownian 
snake, one has, $\N_x$-a.e. for every $i\in I$ and every
$s\in[a_i,b_i]$, 
$$\tau_D(W_s)=\tau_D(W_{a_i})=\zeta_{a_i},$$
and more precisely all paths $W_s$, $s\in[a_i,b_i]$ coincide up to
their exit time from $D$. For every $i\in I$, we then define an element
$W^{(i)}$ of $C(\R_+,\mathcal{W})$ by setting, for every $s\geq 0$,
$$W^{(i)}_s(t) := W_{(a_i+s)\wedge b_i}(\zeta_{a_i}+t),\quad \hbox{for }
0\leq t\leq \zeta_{(W^{(i)}_s)}:=\zeta_{(a_i+s)\wedge b_i}-\zeta_{a_i}.$$
Informally, the $W^{(i)}$'s represent the ``excursions'' of the Brownian snake outside $D$
(the word ``outside'' is a little misleading here, because although these excursions start from a point of $\partial D$, they will 
typically come back inside $D$).

We also need to introduce a $\sigma$-field that contains the information about the 
paths $W_s$ before they exit $D$. To this end, we set, for every $s\geq 0$,
$$\eta^D_s:=\inf\{r\geq 0: \int_0^r \mathrm{d}u\,\mathbf{1}_{\{\zeta_u \leq \tau_D(W_u)\}} > s\},$$
and we let $\mathcal{E}^D$ be the $\sigma$-field generated by the
process $(W_{\eta^D_s})_{s\geq 0}$ and the class of all sets 
that are $\N_x$-negligible. The random measure $\z^D$ is measurable with respect to $\mathcal{E}^D$ (see 
\cite[Proposition 2.3]{LG0}).

We now state the special Markov property \cite[Theorem 2.4]{LG0}.

\begin{proposition}
\label{SMP}
Under $\N_x$, conditionally on $\mathcal{E}^D$, the point measure
$$\sum_{i\in I} \delta_{W^{(i)}}$$
is Poisson with intensity 
$$\int \z^D(\mathrm{d}y)\,\N_y.$$
\end{proposition}

\noindent{\bf Remarks.} (i) Since on the event $\{\z^D=0\}$ there are no excursions outside $D$, the previous 
proposition is equivalent to the same statement where $\N_x$ is replaced by 
the probability measure $\N_x(\cdot\mid \z^D\not =0)$. 

\noindent{(ii)} In what follows we will apply the special Markov property 
in a conditional form. Suppose that $D=(a,\infty)$ for some $a>0$ and that $x>a$. Then
the preceding statement remains valid if we replace $\N_x$ by $\N_x(\cdot\cap\{ \r\subset(0,\infty)\})$,
provided we also replace $\int \z^D(\mathrm{d}y)\,\N_y$ by $\int \z^D(\mathrm{d}y)\,\N_y(\cdot\cap\{\r\subset(0,\infty)\})$. This follows from the fact that conditioning a Poisson point measure 
on having no point on a set of finite intensity is equivalent to removing the points that fall into
this set. We omit the details.

\medskip
For $a<x$, we write $\z_{a}:=\langle \z^{(a,\infty)},1 \rangle$ for the  total mass of the exit measure  outside $(a,\infty)$. We will use the Laplace transform of $\z_a$ under $\N_x$, which is 
given by
\begin{equation}
\label{Laplaceexit}
\N_x\Big(1-\exp(-\mu \z_a)\Big) = \frac{1}{\Big(\mu^{-1/2} + \sqrt{\frac{2}{3}}\,(x - a)\Big)^2},
\end{equation}
for every $\mu\geq 0$. This formula is easily derived from the fact that the (nonnegative) function 
$u(x)=\N_x(1-\exp(-\mu \z_a))$ defined for $x\in(a,\infty)$ solves the differential equation
$u''=4u^2$ with boundary conditions $u(a)=\mu$ and $u(\infty)=0$ (see \cite[Chapter V]{LGZ}). 
On the other hand, an application of the special Markov property shows that, for
every $b<a<x$,
$$\N_x(\exp(-\lambda \z_b)\mid \mathcal{E}^{(a,\infty)}) = \exp\Big(- \z_a\,\N_a(1-\exp(-\lambda \z_b))\Big).$$
If we substitute formula \eqref{Laplaceexit} in the last display, and compare with
\eqref{LaplaceCSBP}, we easily get that the process $(\z_{x-a})_{a> 0}$ is Markov under $\N_x$, with the transition kernels of the continuous-state branching process with branching mechanism 
$\psi(u)=\sqrt{8/3} \,u^{3/2}$. Although $\N_x$ is an infinite measure, the preceding assertion makes sense,
simply because we can restrict our attention to the finite measure event $\{\z_{x-\ve}>0\}$, for
any choice of $\ve>0$.
It follows that $(\z_{x-a})_{a>0}$ has a c\`adl\`ag modification under $\N_x$, which we
consider from now on. 

\medskip
We finally explain an extension of the special Markov property where we consider
excursions outside a random domain. For definiteness, we fix $x=0$, and for
every $a>0$, we set $\mathcal{E}_a=\mathcal{E}^{(-a,\infty)}$. Let $H$ be a 
random variable with values in $(0,\infty]$, such that $\N_0(H<\infty)<\infty$, and assume that $H$ is a stopping time of the filtration $(\mathcal{E}_a)_{a> 0}$
in the sense that, for every $a>0$, the event $\{H\leq a\}$ is $\mathcal{E}_a$-measurable. 
As usual we can  define the $\sigma$-field $\mathcal{E}_H$ that consists of all
events $A$ such that $A\cap\{H\leq a\}$ is $\mathcal{E}_a$-measurable, for every
$a>0$. Since $\z_{-a}$ is $\mathcal{E}_a$-measurable for every $a>0$, it
follows by standard arguments that the random variable $\z_{-H}$ is 
$\mathcal{E}_H$-measurable (at this point it is important that we have taken a
 c\`adl\`ag modification of the process $(\z_{-a})_{a>0}$).
 
We may
consider the excursions $(W^{H,(i)})_{i\in I}$ of the Brownian snake 
outside $(-H,\infty)$. These excursions are defined in exactly the same way as in 
the case where $H$ is deterministic, considering now the connected components
of the open set $\{s\geq 0: W_s(t)<-H\hbox{ for some }t\in[0,\zeta_s]\}$. We define
$\wt W^{H,(i)}$ by shifting $W^{H,(i)}$ so that it starts from $0$.

\begin{proposition}
\label{SMPfort}
Under the probability measure $\N_0(\cdot\mid H<\infty)$, conditionally on the $\sigma$-field 
$\mathcal{E}_{H}$,
the point measure 
$$\sum_{i\in I} \delta_{\wt W^{H,(i)}}$$
is Poisson with intensity 
$$\z_{-H}\,\N_0.$$
\end{proposition}
This proposition can be obtained by arguments very similar
to the derivation of the strong Markov property of Brownian motion 
from the simple Markov property: we approximate $H$ 
with stopping times greater than $H$ that take only countably many values, 
then use Proposition \ref{SMP}
and finally perform a suitable passage to the limit. We leave the
details to the reader. 

\medskip
\noindent{\bf The Brownian snake and super-Brownian motion.} The initial motivation for
studying the Brownian snake came from its connection with super-Brownian motion, which we
briefly recall. Under the excursion measure $\N_x(\rd \omega)$, the lifetime process 
$(\zeta_s(\omega))_{s\geq 0}$ is distributed as a Brownian excursion, and so we 
can define for every $t\geq 0$ the local time proces $(\ell^t_s(\omega))_{s\geq 0}$ 
of this excursion at level $t$. Next let $\mu$ be a finite measure on $\R$, and 
let 
$$\mathcal{N}(\rd \omega)=\sum_{k\in K} \delta_{\omega_{(k)}}(\rd \omega)$$ 
be a Poisson measure on $C(\R_+,\mathcal{W})$ with intensity $\int \mu(\rd x)\,\N_x(\rd \omega)$. 
For every $t>0$, let $\x_t$ be the random measure on $\R$ defined by setting, for every
nonnegative measurable function $\varphi$ on $\R$,
\begin{equation}
\label{SBMsnake}
\langle \x_t,\varphi\rangle =\sum_{k\in K} \int_0^{\sigma(\omega_{(k)})} \rd \ell^t_s(\omega_{(k)})\,\varphi(\wh W_s(\omega_{(k)})).
\end{equation}
If we also set $\x_0=\mu$, the process $(\x_t)_{t\geq0}$ is then a super-Brownian motion with branching mechanism
$\psi_0(u)=2u^2$ started from $\mu$ (see \cite[Theorem IV.4]{LGZ}). A nice feature of this construction is the
fact that it also gives the associated historial process: Just consider for every $t>0$ the
random measure $\mathbf{X}_t$ defined by setting 
\begin{equation}
\label{historicalsnake}
\langle \mathbf{X}_t,\Phi\rangle =\sum_{k\in K} \int_0^{\sigma(\omega_{(k)})} \rd \ell^t_s(\omega_{(k)})\,\Phi(W_s(\omega_{(k)})),
\end{equation}
for every nonnegative measurable function $\Phi$ on $\mathcal{W}$. 
Some of the forthcoming results are stated in terms of super-Brownian motion and its historical process. Without
loss of generality we may and will assume that these processes are obtained by formulas \eqref{SBMsnake} and
\eqref{historicalsnake} of the previous construction. This also means that we consider the special branching 
mechanism $\psi_0(u)=2u^2$, but of course the case of a general quadratic branching mechanism can then  
be handled via scaling arguments.

\section{The Brownian plane}

\subsection{The Brownian plane as a random metric space}

We start by giving a characterization of the Brownian plane as a random pointed metric space
satisfying appropriate properties. We let
$\K_{bcl}$ denote the space of all isometry classes of pointed boundedly
compact length spaces. The space $\K_{bcl}$ is equipped with the local
Gromov-Hausdorff distance $\rd_{LGH}$ (see \cite[Section 2.1]{CLG}) and is a 
Polish space, that is, separable and complete for this distance. For $r>0$
and $F\in \K_{bcl}$, we use the notation $B_r(F)$ for the closed
ball of radius $r$ centered at the distinguished point of $F$. Note that
$B_r(F)$ is always viewed as a {\it pointed} compact metric space.

The Brownian 
plane $\mathcal{P}_\infty$ is then a random variable taking values
in the space $\K_{bcl}$. 

\begin{definition}
\label{local-isometry}
Let $E_1$ and $E_2$ be two random variables with values in $\K_{bcl}$.
We say that $E_1$ and $E_2$ are locally isometric if, for every $\delta>0$, there exists a 
number $r>0$ and a
coupling of $E_1$ and $E_2$ such that the balls $B_r(E_1)$ and $B_r(E_2)$ 
are isometric with probability at least $1-\delta$.
\end{definition}

We leave it to the reader to verify that this is an equivalence relation (only transitivity
is not obvious).
The interest of this definition comes from the next proposition. If
$E$ is a (random) metric space and $\lambda>0$, we use the
notation $\lambda\cdot E$ for the same metric space where 
the distance has been multiplied by $\lambda$. 

\begin{proposition}
\label{caract-BP}
The distribution of the Brownian plane is characterized in the set of 
all probability measures on $\K_{bcl}$ by the following
two properties:
\begin{enumerate}
\item[{\rm (i)}] The Brownian plane is locally isometric to the Brownian map.
\item[{\rm (ii)}] The Brownian plane is scale invariant, meaning that
$\lambda\cdot \mathcal{P}_\infty$ has the same distribution as $\mathcal{P}_\infty$,
for every $\lambda >0$. 
\end{enumerate}
\end{proposition}

\begin{proof} 
The fact that property (i) holds is Theorem 1 in \cite{CLG}. Property (ii) is immediate from
the construction in \cite{CLG}, or directly from the convergence (1) in \cite[Theorem 1]{CLG}.
So we just have to prove that these two properties characterize the distribution of the Brownian 
plane. Let $E$ be a random variable with values in $\K_{bcl}$, which is both 
locally isometric to the Brownian map and scale invariant. Then, $E$ is also
locally isometric to the Brownian plane, and, for every $\delta>0$ we can find $r>0$
and a coupling of $E$ and $\mathcal{P}_\infty$ such that
$$P[B_r(E)=B_r(\mathcal{P}_\infty)]>1-\delta,$$
where the equality is in the sense of isometry between pointed compact metric spaces. Trivially this
implies that, for every $a>0$,
$$P[B_a(\frac{a}{r}\cdot E)=B_a(\frac{a}{r}\cdot\mathcal{P}_\infty)]>1-\delta.$$
By scale invariance, $\frac{a}{r}\cdot E$ and $\frac{a}{r}\cdot\mathcal{P}_\infty$ have the same 
distribution as $E$ and $\mathcal{P}_\infty$ respectively. So we get that for every 
$\delta>0$, for every $a>0$, we can find a coupling of $E$ and $\mathcal{P}_\infty$ such that
$$P[B_a(E)=B_a(\mathcal{P}_\infty)]>1-\delta.$$
Recalling the definition of the local Gromov-Hausdorff distance $\rd_{LGH}$ (see e.g. \cite[Section 2.1]{CLG})
we obtain that, for every $\ve>0$ and every $\delta>0$, there exists a coupling
of $E$ and $\mathcal{P}_\infty$ such that 
$$P[\rd_{LGH}(E,\mathcal{P}_\infty)<\ve] > 1-\delta.$$
Clearly this implies that the L\'evy-Prokhorov distance between the distributions
of $E$ and $\mathcal{P}_\infty$ is $0$ and thus $E$ and $\mathcal{P}_\infty$ have the same distribution.
\end{proof}

\subsection{A new construction of the Brownian plane}
\label{newrep}

In this section, we provide a construction of the Brownian plane, which is different
from the one in \cite{CLG}. We then use Proposition \ref{caract-BP} and Theorem \ref{decotheo} to
prove the equivalence of the two constructions.

We consider a nine-dimensional Bessel process $R=(R_t)_{t\geq 0}$ starting from $0$ and,
conditionally on $R$, two independent Poisson point measures $\n'(\rd t,\rd \omega)$ and $\n''(\rd t,\rd \omega)$ on $\R_+\times C(\R_+,\mathcal{W})$ with 
the same intensity
$$2\,\ind{\{\r(\omega)\subset(0,\infty)\}}\,\rd t\,\N_{R_t}(\rd \omega).$$
It will be convenient to write
$$\n'=\sum_{i\in I} \delta_{(t_i,\omega^i)}\;,\quad \n''=\sum_{i\in J} \delta_{(t_i,\omega^i)},
$$
where the indexing sets $I$ and $J$ are disjoint.

We also consider the sum $\n=\n'+\n''$, which conditionally on $R$ is Poisson with intensity 
$$4\,\ind{\{\r(\omega)\subset(0,\infty)\}}\,\rd t\,\N_{R_t}(\rd \omega),$$
and we have 
\begin{equation}
\label{Poisson}
\n=\sum_{i\in I\cup J} \delta_{(t_i,\omega^i)}.
\end{equation}

We start by introducing the infinite random tree that will be crucial in our construction of the Brownian plane. For every 
$i\in I\cup J$,
write $\sigma_i=\sigma(\omega^i)$ and let $(\zeta^i_s)_{s\geq 0}$ be the lifetime process associated with $\omega^i$. Then the function 
$(\zeta^i_s)_{0\leq s\leq \sigma_i}$ codes a rooted compact real tree, which is denotes by $\t^i$,
and we write $p_{\zeta^i}$ for the canonical projection from $[0,\sigma_i]$ onto $\t^i$
(see e.g. \cite[Section 2]{LGtree} for basic facts about the coding of trees by continuous functions). We construct a 
random non-compact real tree $\t_\infty$ by grafting to the half-line $[0,\infty)$
(which we call the ``spine'') the tree $\t^i$ at point $t_i$, for every
$i\in I\cup J$. Formally, the tree $\t_\infty$ is obtained from the disjoint union
$$[0,\infty) \cup\Bigg(\bigcup_{i\in I\cup J} \t^i\Bigg)$$
by identifying the point $t_i$ of $[0,\infty)$ with the root $\rho_i$ of $\t^i$, for every $i\in I\cup J$. The metric $d_\infty$
on $\t_\infty$ is determined as follows. The restriction of $d_\infty$ to each tree $\t^i$ is (of course) the metric $d_{\t^i}$ on 
$\t^i$. If $x\in \t^i$ and $t\in[0,\infty)$, we take $d_\infty(x,t)= d_{\t^i}(x,\rho_i)+ |t_i-t|$. If $x\in\t^i$ and $y\in\t^j$, with $i\not = j$,
we take $d_\infty(x,y)= d_{\t^i}(x,\rho_i) + |t_i-t_j| + d_{\t^j}(\rho_j,y)$. By convention, $\t_\infty$
is rooted at $0$. The infinite tree $\t_\infty$ is equipped with a volume measure $\mathbf{V}$, which puts no
mass on the spine and whose restriction to each tree $\t^i$ is the natural volume measure 
on $\t^i$ defined as the image of Lebesgue measure on $[0,\sigma_i]$ under the projection $p_{\zeta^i}$.

We also define labels on the tree $\t_\infty$. The label $\Lambda_x$
of a vertex $x\in\t_\infty$ is defined by $\Lambda_x= R_t$ if $x=t$ belongs to the spine $[0,\infty)$, and 
$\Lambda_x= \wh \omega^i_s$ if $x=p_{\zeta^i}(s)$ belongs to the subtree $\t^i$, for some 
$i\in I\cup J$. Note that the mapping $x\mapsto \Lambda_x$ is continuous almost surely. 
For future use, we also notice that, if $x=p_{\zeta^i}(s)$ belongs to the subtree $\t^i$, 
the quantities $\omega^i_s(t)$, $0\leq t\leq \zeta^i_s$ are the labels of the ancestors of $x$
in $\mathcal{T}^i$. 

We will use the fact that labels are ``transient'' in the sense of the following lemma. Recall the notation
$\omega_*=W_*(\omega)$.

\begin{lemma}
\label{transience-lemma}
We have a.s.
$$\lim_{r \uparrow\infty}\Big(\inf_{i\in I\cup J, t_i>r} \omega^i_*\Big)= +\infty.$$
\end{lemma}

\begin{proof}
It is enough to verify that, for every $A>0$, we have
$$\lim_{r\uparrow\infty} P\Big( \inf_{i\in I\cup J,\, t_i\geq r} \omega^i_* < A\Big) =0.$$
However by construction,
\begin{align*}
&P\Big( \inf_{i\in I\cup J, t_i\geq r} \omega^i_* < A\Big)\\
&\quad =P\Big( \inf_{t\geq r} R_t<A\Big) + E\Big[\mathbf{1}\Big\{  \inf_{t\geq r} R_t\geq A\Big\}
\Big(1-\exp\Big(-4\int_r^\infty \rd t \,\N_{R_t}(0<W_*< A)\Big)\Big)\Big]\\
&\quad= P\Big( \inf_{t\geq r} R_t<A\Big) + E\Big[\mathbf{1}\Big\{  \inf_{t\geq r} R_t\geq A\Big\}
\Big(1-\exp\Big(-6\int_r^\infty \rd t \big(\frac{1}{(R_t-A)^2} -\frac{1}{(R_t)^2}\big)\Big)\Big)\Big],
\end{align*}
using \eqref{hittingzero}. The desired result easily follows from the fact that the integral $\int^\infty \rd t\,(R_t)^{-3}$ is convergent.
\end{proof}

Until now, we have not used the fact that $\n$ is decomposed in the form $\n=\n'+\n''$. This decomposition 
corresponds intuitively to the fact that the trees $\t^i$ are grafted on the left side of the spine $[0,\infty)$
when $i\in I$, and on the right side when $i\in J$. We make this precise by defining an exploration process 
of the tree. To begin with, we define, for every $u\geq 0$,
$$\tau'_u:= \sum_{i\in I} \ind{\{t_i\leq u\}}\,\sigma_i\;,\quad \tau''_u:= \sum_{i\in J} \ind{\{t_i\leq u\}}\,\sigma_i\;.$$
Note that both $u\mapsto \tau'_u$ and $u\mapsto \tau''_{u}$ are nondecreasing and right-continuous. The left limits
of these functions are denoted by
$\tau'_{u-}$ and $\tau''_{u-}$ respectively, and $\tau'_{0-}=\tau''_{0-}=0$ by convention. 

Then, for every $s\geq 0$, there is a unique $u\geq 0$, such that $\tau'_{u-}\leq s\leq \tau'_{u}$, and:
\begin{enumerate}
\item[$\bullet$] Either there is a (unique) $i\in I$ such that $u=t_i$, and we set
$$\Theta'_s:= p_{\zeta^i}(s-\tau'_{t_i-})
.$$
\item[$\bullet$] Or there is no such $i$ and we set $\Theta'_s=u$
.
\end{enumerate}
We define similarly $(\Theta''_s)_{s\geq 0}$
 by replacing $(\tau'_u)_{u\geq 0}$ by $(\tau''_u)_{u\geq 0}$ and $I$ by $J$.
Informally, $(\Theta'_s)_{s\geq 0}$ and $(\Theta''_s)_{s\geq 0}$ correspond to the exploration of respectively the left and the right side
of the tree $\t_\infty$. 
Noting that $\Theta'_0=\Theta''_0=0$, we define $(\Theta_s)_{s\in \R}$ 
by setting
$$\Theta_s:=\left\{\begin{array}{ll}
\Theta'_s\quad&\hbox{if }s\geq 0,\\
\Theta''_{-s}\quad&\hbox{if }s\leq 0.
\end{array}
\right.
$$
It is straightforward to verify that the mapping $s\mapsto \Theta_s$ is continuous.
We also note that the volume measure $\mathbf{V}$ on $\t_\infty$ is  the image of Lebesgue
measure on $\R$ under the mapping $s\mapsto \Theta_s$.

This exploration process allows us to define intervals on
$\t_\infty$. 
Let us make the convention that, if $s>t$, the ``interval'' $[s,t]$ is defined by $[s,t]=[s,\infty)\cup (-\infty,t]$. 
Then, for every $x,y\in\t_\infty$, there is a smallest interval $[s,t]$, with $s,t\in\R$, such that
$\Theta_s=x$ and $\Theta_t=y$, and we define  
$$[x,y]:=\{\Theta_r:r\in[s,t]\}.$$
Note that $[x,y]\not =[y,x]$ unless $x=y$. We may now turn to our construction
of the Brownian plane. We set, for every $x,y\in\t_\infty$,
\begin{equation}
\label{Dzero}
D^\circ_\infty(x,y)= \Lambda_x + \Lambda_y - 2\max\Bigg(\min_{z\in [x,y]} \Lambda_z,
\min_{z\in [y,x]} \Lambda_z\Bigg),
\end{equation}
and then 
\begin{equation}
\label{Dinfty}
D_\infty(x,y) = \inf_{x_0=x,x_1,\ldots,x_p=y} \sum_{i=1}^p D^\circ_\infty(x_{i-1},x_i)
\end{equation}
where the infimum is over all choices of the integer $p\geq 1$ and of the
finite sequence $x_0,x_1,\ldots,x_p$ in $\t_\infty$ such that $x_0=x$ and
$x_p=y$. Note that we
have 
\begin{equation}
\label{easybd}
D^\circ_\infty(x,y)\geq D_\infty(x,y)\geq |\Lambda_x -\Lambda_y|,
\end{equation}
for every $x,y\in\t_\infty$. Furthermore, it
is immediate from our definitions that 
$$D_\infty(0,x)=D^\circ_\infty(0,x)=\Lambda_x$$
for every $x\in \t_\infty$. As a consequence of the continuity of the mapping $s\mapsto \Lambda_{\Theta_s}$, we have
$D^\circ_\infty(x_0,x)\la 0$ (hence also $D_\infty(x_0,x)\la 0$) as $x\to x_0$, for every $x_0\in \t_\infty$. 

It is not hard to verify that $D_\infty$ is a pseudo-distance on $\t_\infty$.
We put $x\approx y$ if and only if $D_\infty (x,y)=0$ and we introduce the quotient space $\wt{\mathcal P}_\infty=\t_\infty /\approx$, which is
equipped with the metric induced by $D_\infty$ and with the distinguished 
point which is the equivalence class of $0$. The volume measure on
$\wt{\mathcal P}_\infty$ is the image of the volume measure $\mathbf{V}$
on $\t_\infty$ under the canonical projection. 

\begin{theorem}
\label{newBP}
The pointed metric space $\wt{\mathcal P}_\infty$ is locally isometric to the Brownian map and 
scale invariant. Consequently, $\wt{\mathcal P}_\infty$ is distributed as the Brownian plane $\mathcal{P}_\infty$.
\end{theorem}

\begin{proof} The fact that $\wt{\mathcal P}_\infty$ is scale invariant is easy from our construction. 
Hence the difficult part of the proof is to verify that $\wt{\mathcal P}_\infty$ is locally isometric to the Brownian map.
Let us start by briefly recalling the construction of the Brownian map $\bm_\infty$. We 
argue under the conditional excursion measure
$\N_0^{(1)} =\N_0(\cdot\mid \sigma =1)$.
Under $\N_0^{(1)}$, the lifetime process $(\zeta_s)_{0\leq s\leq 1}$ is
a normalized Brownian excursion, and the tree $\t_\zeta$ coded by $(\zeta_s)_{0\leq s\leq 1}$
is the so-called CRT. As previously, $p_\zeta$ stands for the canonical projection from $[0,1]$
onto $\t_\zeta$. We can define intervals on $\t_\zeta$ in a way analogous to what we did before for $\mathcal{T}_\infty$:
If $x,y\in\t_\zeta$, $[x,y]=\{p_\zeta(r):r\in[s,t]\}$, where $[s,t]$ is the smallest interval
such that $p_\zeta(s)=x$ and $p_\zeta(t)=y$, using now the convention that
the interval  $[s,t]$ is defined by $[s,t]=[s,1]\cup[0,t]$ when $s>t$. Then we equip $\t_\zeta$ with Brownian labels
by setting $\Gamma_x=\wh W_s$ if $x=p_\zeta(s)$. For every $x,y\in\t_\zeta$, we define $D^\circ(x,y)$, resp.~$D(x,y)$, by exactly the same formula as in \eqref{Dzero}, resp.~\eqref{Dinfty}, replacing $\Lambda$ by $\Gamma$. 
We have again the bound $D(x,y)\geq |\Gamma_x-\Gamma_y|$. 
We then observe that
$D$ is a pseudo-distance on $\t_\zeta$, and the Brownian map $\bm_\infty$ is the associated 
quotient metric space. The distinguished point of $\bm_\infty$ is chosen as the (equivalence class of the) 
vertex $x_\bm$ of $\t_\zeta$ with minimal label, and we note that $D(x_\bm,x)= \Gamma_x-\Gamma_{x_\bm}=
\Gamma_x-W_*$ for every
$x\in\t_\zeta$. 

If we replace the normalized Brownian excursion by a Brownian excursion with duration $r>0$, that
is, if we argue under $\N_0^{(r)}$,
and perform the same construction, simple scaling arguments show that the resulting pointed 
metric space is distributed as $r^{1/4}\cdot \bm_\infty$ and is thus locally isometric
to $\bm_\infty$ (both are locally isometric to the Brownian plane). Consequently, under the
probability measure
$$\N_0(\cdot \mid \sigma >1) =\int_0^\infty \frac{\rd r}{2\sqrt{2\pi r^3}}\, \N^{(r)}_0(\cdot)$$
the preceding construction also yields a random pointed 
metric space which is locally isometric
to $\bm_\infty$. Let us write $\mathbf{M}$ for this random pointed metric space. We will argue that $\mathbf{M}$ is locally isometric to
$\wt{\mathcal P}_\infty$, which will complete the proof. Some of the arguments that follow are similar to those used in
\cite[Proof of Proposition 4]{CLG} to verify that the Brownian plane is locally isometric to the Brownian map.

We set for every $b>0$, 
$$A_b:=\int_0^\sigma \rd s\,\ind{\{\tau_{(-b,\infty)}(W_s)<\infty\}},$$
where we used the notation $\tau_D(\w)$ introduced in subsection \ref{prelisnake}. Still 
with the notation of this subsection, the random variable $A_b$ is $\mathcal{E}_b$-measurable, and 
it follows that 
$$H:=\inf\{b\geq 0: A_b =1\}$$
is a stopping time of the filtration $(\mathcal{E}_a)_{a>0}$. Observe that $\{H<\infty\}=\{\sigma>1\}$,
$\N_0$ a.e. From Proposition \ref{SMPfort}, we get that under the probability measure
$\N_0(\cdot \mid \sigma >1)$, and conditionally on the pair $(H,\z_{-H})$, the excursions
of the Brownian snake outside $(-H,\infty)$ form a Poisson point process with
intensity $\z_{-H}\,\N_{-H}$ (incidentally this also implies that $\z_{-H}>0$ a.e. on $\{\sigma >1\}$). Among the
excursions outside $(-H,\infty)$, there is exactly one that attains the minimal value $W_*$, and
conditionally on $H=h$ and $W_*=a$ (with $a<-h$), this excursion is distributed according to
$\N_{-h}(\cdot \mid W_*=a)$. 
 
Now compare Theorem \ref{decotheo} with the construction of $\wt{\mathcal P}_\infty$ given 
above to see that we can find a coupling of the Brownian snake under $\N_0(\cdot\mid \sigma >1)$
and of the triplet $(R,\n',\n'')$ determining the labeled tree $(\t_\infty, (\Lambda_x)_{x\in\t_\infty})$, in such a way that the following
properties hold. There exists a (random) real $\delta>0$ and an isometry $\mathcal{I}$ from
the ball $B_\delta(\t_\zeta)$ (centered at the distinguished vertex $x_\bm=p_\zeta(s_\bm)$) 
onto the ball $B_\delta(\t_\infty)$ (centered at $0$). This isometry preserves intervals, in the sense that if $x, y \in B_\delta(\t_\zeta)$, 
$\mathcal{I}([x,y] \cap B_\delta(\t_\zeta))= [\mathcal{I}(x),\mathcal{I}(y)] \cap B_\delta(\t_\infty)$. 
Furthermore, the isometry $\mathcal{I}$ preserves labels up to a shift by $-W_*$,
meaning that $\Lambda_{\mathcal{I}(x)}= \Gamma_{x} - W_*$ for every $x\in B_\delta(\t_\zeta)$. 
Consequently, we have 
$$D(x_\bm,x)= \Gamma_{x} - W_*=\Lambda_{\mathcal{I}(x)}= D_\infty(0,\mathcal{I}(x))$$ 
for every $x\in B_\delta(\t_\zeta)$.

Next we can choose $\eta>0$ small enough so that labels on $\t_\zeta\backslash B_\delta(\t_\zeta)$
are all strictly larger than $W_*+ 2\eta$ and labels on $\t_\infty \backslash B_\delta(\t_\infty)$ 
are all strictly larger than $2\eta$ (we use
Lemma \ref{transience-lemma} here). In particular, if $x\in \t_\zeta$, the condition 
$D(x_\bm,x)\leq 2\eta$ implies that $x\in B_\delta(\t_\zeta)$ and, if
$x'\in\t_\infty$, the condition $D_\infty(0,x')\leq 2\eta$ implies that $x'\in B_\delta(\t_\infty)$. We claim that
\begin{equation}
\label{claimdist}
D(x,y)=D_\infty(\mathcal{I}(x),\mathcal{I}(y)),
\end{equation}
 for every
$x,y\in \t_\zeta$ such that $D(x_\bm,x)\leq \eta$ and $D(x_\bm,y)\leq \eta$. 
To verify this claim, first note that, if  $x',y'\in \t_\infty$ are such that $D_\infty(0,x')=\Lambda_{x'}\leq 2\eta$ and $D_\infty(0,y')=\Lambda_{y'}\leq 2\eta$,
we can compute $D^\circ_\infty(x',y')$ using formula \eqref{Dzero}, and in the
right-hand side of this formula we may replace the interval $[ x',y']$ by $[ x',y']\cap B_\delta(\t_\infty)$
(because obviously the minimal value of $\Lambda$ on $[ x',y']$ is attained on $[ x',y']\cap B_\delta(\t_\infty)$). 
A similar replacement may be made in the analogous formula for $D^\circ(x,y)$ when $x,y\in \t_\zeta$
are such that $\Gamma_{x}\leq W_*+2\eta$ and $\Gamma_{y}\leq W_*+2\eta$. Using the isometry $\mathcal{I}$,
we then obtain that 
\begin{equation}
\label{claimdist1}
D^\circ(x,y)=D^\circ_\infty(\mathcal{I}(x),\mathcal{I}(y))
\end{equation}
 for every
$x,y\in \t_\zeta$ such that $D(x_\bm,x)\leq 2\eta$ and $D(x_\bm,y)\leq 2\eta$.
Then, let $x',y'\in \t_\infty$ be such that $\Lambda_{x'}\leq \eta$ and $\Lambda_{y'}\leq \eta$.
If we use formula \eqref{Dinfty} to evaluate $D_\infty(x',y')$, we may in the right-hand side
of this formula
 restrict our attention
to ``intermediate'' points $x_i$  whose label $\Lambda_{x_i}$ is smaller than $2\eta$
(indeed if one of the intermediate points has a label strictly greater than $2\eta$, the
sum in the right-hand side of \eqref{Dinfty} will be strictly greater than $2\eta\geq D_\infty(x',y')$,
thanks to \eqref{easybd}). A similar observation
holds if we use the analog of \eqref{Dinfty} to compute $D(x,y)$ when 
$x,y\in \t_\zeta$ are such that $D(x_\bm,x)\leq \eta$ and $D(x_\bm,y)\leq \eta$. 
Our claim \eqref{claimdist} is a consequence of the preceding considerations and \eqref{claimdist1}. 

 It follows from \eqref{claimdist} that
$\mathcal{I}$ induces an isometry from the ball $B_\eta(\mathbf{M})$ onto
the ball $B_\eta(\wt{\mathcal P}_\infty)$. This implies that $\mathbf{M}$ is locally isometric 
to $\wt{\mathcal P}_\infty$, and the proof is complete.
\end{proof}

In view of Theorem \ref{newBP}, we may and will write $\mathcal{P}_\infty$ instead of $\wt{\mathcal P}_\infty$ 
for the random metric space that we constructed in the first part of this subsection. We denote the canonical projection from $\t_\infty$ onto $\mathcal{P}_\infty$
by $\Pi$. The fact that $D_\infty(x_0,x)\la 0$ as $x\to x_0$, for every fixed $x_0\in\t_\infty$, shows that 
$\Pi$ is continuous. The argument of the preceding proof makes it possible to transfer  several known properties of the Brownian map to the
space $\mathcal{P}_\infty$. First, for every
$x,y\in \t_\infty$, we have 
$$D_\infty(x,y)=0\quad\hbox{if and only if}\quad D^\circ_\infty(x,y)=0.$$
Indeed this property will hold for $x$ and $y$ belonging to a sufficiently small
ball centered at $0$ in $\t_\infty$, by \cite[Theorem 3.4]{LGTopo} and the coupling argument explained in the preceding proof. 
The scale invariance of the Brownian plane then completes the argument. Similarly, we have the so-called
``cactus bound'', for every $x,y\in\t_\infty$ and every continuous path $(\gamma(t))_{0\leq t\leq 1}$
in $\pp_\infty$ such that $\gamma(0)=\Pi(x)$ and $\gamma(1)=\Pi(y)$,
\begin{equation}
\label{cactus-bound}
\min_{0\leq t\leq 1} D_\infty(0,\gamma(t))\leq \min_{z\in\llbracket x,y\rrbracket} \Lambda_z,
\end{equation}
where $\llbracket x, y\rrbracket$ stands for the geodesic segment between $x$
and $y$ in the tree $\t_\infty$. The bound \eqref{cactus-bound} follows from the
analogous result for the Brownian map \cite[Proposition 3.1]{LGGeo} and the coupling 
argument of the preceding proof.

Since labels correspond to distances from the distinguished point, we have, for every $r>0$,
$$B_r(\mathcal{P}_\infty)
= \Pi\Big( \{x\in\t_\infty: \Lambda_x \leq r\}\Big).$$
Recall the definition of the hull $B_r^\bullet(\mathcal{P}_\infty)$ in Section 1.
We have 
\begin{equation}
\label{formula-hull}
B_r^\bullet(\mathcal{P}_\infty)
= \pp_\infty\;\backslash\; \Pi\Big(\{x\in\t_\infty: \Lambda_y > r,\;\forall y\in\llbracket x,\infty\llbracket\}\Big),
\end{equation}
where $\llbracket x,\infty\llbracket$ is the geodesic path 
from $x$ to $\infty$ in the tree $\t_\infty$. The fact that $B_r^\bullet(\mathcal{P}_\infty)$ is contained
in the right-hand side of \eqref{formula-hull} is easy: If $x\in\t_\infty$ is such that $\Lambda_y > r$
for every $y\in\llbracket x,\infty\llbracket$, then $\Pi(\llbracket x,\infty\llbracket)$ gives a
continuous path going from $\Pi(x)$ to $\infty$ and staying outside the ball $B_r(\mathcal{P}_\infty)$.
Conversely, suppose that $x\in\t_\infty$ is such that
$$\min_{y\in \llbracket x,\infty\llbracket} \Lambda_y \leq r.$$
Then, if $(\gamma(t))_{t\geq 0}$ is any continuous path going from
$\Pi(x)$ to $\infty$ in $\pp_\infty$, the bound \eqref{cactus-bound} leads to
$$\min_{t\geq 0} D(0,\gamma(t))\leq \min_{y\in \llbracket x,\infty\llbracket} \Lambda_y \leq r,$$
and it follows that $\Pi(x)\in B_r^\bullet(\mathcal{P}_\infty)$. 

Write $\partial B_r^\bullet(\mathcal{P}_\infty)$ for the topological boundary
of $B_r^\bullet(\mathcal{P}_\infty)$. It follows from \eqref{formula-hull} that
\begin{equation}
\label{boundary-hull}
\partial B_r^\bullet(\mathcal{P}_\infty)
=  \Pi\Big(\{x\in\t_\infty: \Lambda_x=r\hbox{ and } \Lambda_y > r,\;\forall y\in\,\rrbracket x,\infty\llbracket\}\Big),
\end{equation}
with the obvious notation $\rrbracket x,\infty\llbracket$. The latter 
formula motivates the definition of the
(generalized) length of the 
boundary of $B_r^\bullet(\mathcal{P}_\infty)$. We observe that
this boundary contains (the image under $\Pi$ of) a single point
on the spine, corresponding to the last visit of $r$ by the process $R$,
$$L_r=\sup\{t\geq 0: R_t=r\}.$$
Any other point
$x\in\t_\infty$ such that $ \Lambda_x=r$
and $\Lambda_y > r$ for every $y\in\,\rrbracket x,\infty\llbracket$ must be of the form
$p_{\zeta^i}(s)$, for some $i\in I\cup J$, with $t_i>L_r$, and some $s\in[0,\sigma_i]$ 
such that the path $\omega^i_s$ hits $r$ exactly at its lifetime. For each fixed 
$i$ (with $t_i>L_r$), the ``quantity''
of such values of $s$ is measured by the total mass $\z_r(\omega^i)$ of the exit measure 
of $\omega^i$ from $(r,\infty)$.
Here  we use the same notation $\z_r=\langle \z^{(r,\infty)},1\rangle$ as previously.

Following the preceding discussion, we define, for every $r>0$,
\begin{equation}
\label{size-boundary}
Z_r:=\int \n(\rd t,\rd \omega)\,\ind{\{L_r<t\}}\,\z_r(\omega)=\sum_{i\in I\cup J, t_i> L_r}\z_r(\omega^i).
\end{equation}
We
observe that the quantities $\z_r(\omega^i)$ in \eqref{size-boundary} are well-defined since
each $\omega^i$ 
is a Brownian snake excursion starting from $R_{t_i}$ and the condition $t_i>L_r$ 
guarantees that $R_{t_i}>r$. 
We interpret $Z_r$ as measuring the size of the boundary of the hull $B_r^\bullet(\mathcal{P}_\infty)$.

Note that at the present stage, it is not clear that the random variable $Z_r$ coincides with the
one introduced in Proposition \ref{approx-exit} (which we  have not yet proved). At the end of subsection 
\ref{exit-process} below, we will verify  that the approximation result of Proposition \ref{approx-exit} holds with
the preceding definition of $Z_r$.

\section{The volume of hulls}

\subsection{The process of boundary lengths}
\label{exit-process}

Our main goal in this subsection is to describe the
distribution of the process $(Z_r)_{r>0}$. We fix $a>0$. 
By formula \eqref{size-boundary} and  the exponential formula for Poisson measures, we have, for every $\lambda\geq 0$,
\begin{equation}
\label{LaplaceEM}
E\Big[ \exp(-\lambda Z_a)]
= E\Big[ \exp \Big(-4\int_{L_a}^\infty \rd t\,\N_{R_t}\Big(\ind{\{\r\subset (0,\infty)\}}(1-e^{-\lambda \z_a})\Big)\Big)\Big].
\end{equation}
The quantity in the right-hand side will be computed via the following two lemmas.

\begin{lemma}
\label{Laplace1}
For every $x>a$ and $\lambda\geq 0$,
$$\N_{x}\Big(\ind{\{\r\subset (0,\infty)\}}(1-e^{-\lambda \z_a})\Big)= 
\frac{3}{2}\Bigg(\Big(x-a +(\frac{2\lambda}{3}+a^{-2})^{-1/2} \Big)^{-2} - x^{-2}\Bigg).$$
\end{lemma}

\begin{proof}
We have
\begin{align*}
\N_{x}\Big(\ind{\{\r\subset (0,\infty)\}}(1-e^{-\lambda \z_a})\Big)
&=\N_{x}\Big(1- \ind{\{\r\subset (0,\infty)\}}e^{-\lambda \z_a}\Big)- \N_{x}\Big(1- \ind{\{\r\subset (0,\infty)\}}\Big)\\
&=\N_{x}\Big(1- \ind{\{\r\subset (0,\infty)\}}e^{-\lambda \z_a}\Big)-\frac{3}{2x^2},
\end{align*}
by \eqref{hittingzero}. In order to compute the first term in the right-hand side, we observe that we have 
$\r\subset (a,\infty)\subset (0,\infty)$ on the event $\{\z_a=0\}$, $\N_x$ a.e., by \eqref{range-exit}. Therefore, we can write
\begin{align*}
\N_{x}\Big(1- \ind{\{\r\subset (0,\infty)\}}e^{-\lambda \z_a}\Big)
&=\N_x\Big(\ind{\{\z_a>0\}}\Big) - \N_{x}\Big(\ind{\{\z_a>0,\r\subset (0,\infty)\}}e^{-\lambda \z_a}\Big)\\
&=\N_x\Big(\ind{\{\z_a>0\}}\Big) - \N_{x}\Big(\ind{\{\z_a>0\}}\,e^{-\lambda \z_a}\,\exp\Big(-\frac{3\z_a}{2a^2}\Big)\Big)\\
&=\N_{x}\Big(1-\exp\Big(-(\lambda +\frac{3}{2a^2})\z_a\Big)\Big).
\end{align*}
In the second equality we used the special Markov property, together with formula \eqref{hittingzero}, to obtain that
the conditional probability of the event $\{\r\subset (0,\infty)\}$ given $\z_a$ is $\exp(-\frac{3\z_a}{2a^2})$. 
The formula of the lemma follows from the preceding two displays and \eqref{Laplaceexit}.
\end{proof}

\begin{lemma}
\label{Laplace2}
For every $\alpha\in(0,a)$,
$$ E\Big[ \exp \Big(6\int_{L_a}^\infty \rd t\,\Big(\frac{1}{(R_t)^2}-\frac{1}{(R_t-\alpha)^2}\Big)\Big)\Big]
= \Big(\frac{a-\alpha}{a}\Big)^3.
$$
\end{lemma}

\begin{proof}
By dominated convergence, we have
$$E\Big[ \exp \Big(6\int_{L_a}^\infty \rd t\,\Big(\frac{1}{(R_t)^2}-\frac{1}{(R_t-\alpha)^2}\Big)\Big)\Big]
=\lim_{b\uparrow\infty} \downarrow E\Big[ \exp \Big(6\int_{L_a}^{L_b} \rd t\,\Big(\frac{1}{(R_t)^2}-\frac{1}{(R_t-\alpha)^2}\Big)\Big)\Big].$$
Let us fix $b>a$. By the time-reversal property of Bessel processes already mentioned after the statement
of Theorem \ref{decotheo}, the 
process $(\wt R_t)_{t\geq 0}$ defined by
$$\wt R_t= R_{(L_b-t)\vee 0}$$
is a Bessel process of dimension $-5$ started from $b$. 
Set $T_a:=\inf\{t\geq 0:\wt R_t=a\}=L_b-L_a$. Write $(B_t)_{t\geq 0}$ for a one-dimensional
Brownian motion which starts from $r$ under the probability measure $P_r$, and for every
$y\in \R$, let $\gamma_y:=\inf\{t\geq 0:B_t=y\}$. Then,
\begin{align*}
E\Big[ \exp \Big(6\int_{L_a}^{L_b} \rd t\,\Big(\frac{1}{(R_t)^2}-\frac{1}{(R_t-\alpha)^2}\Big)\Big)\Big]
&=E\Big[ \exp \Big(6\int_{0}^{T_a} \rd t\,\Big(\frac{1}{(\wt R_t)^2}-\frac{1}{(\wt R_t-\alpha)^2}\Big)\Big)\Big]\\
&=\Big(\frac{b}{a}\Big)^3\,
E_b\Big[ \exp \Big(-6\int_{0}^{\gamma_a}\,\frac{ \rd t}{(B_t-\alpha)^2}\Big)\Big],
\end{align*}
where the last equality is a consequence of the absolute continuity relation found as
Lemma 1 in \cite{Bessel}. Next observe that
$$
E_b\Big[ \exp \Big(-6\int_{0}^{\gamma_a} \frac{ \rd t}{(B_t-\alpha)^2}\Big)\Big]
=E_{b-\alpha}\Big[ \exp \Big(-6\int_{0}^{\gamma_{a-\alpha}}\frac{ \rd t}{(B_t)^2}\Big)\Big]
= \Big(\frac{a-\alpha}{b-\alpha}\Big)^3,$$
where the second equality is well known (and can again be viewed as a consequence of Lemma 1 in \cite{Bessel}). By
combining the last two displays, we get
$$E\Big[ \exp \Big(6\int_{L_a}^{L_b} \rd t\,\Big(\frac{1}{(R_t)^2}-\frac{1}{(R_t-\alpha)^2}\Big)\Big)\Big]
= \Big(\frac{b}{a}\Big)^3\times \Big(\frac{a-\alpha}{b-\alpha}\Big)^3,$$
and the desired result follows by letting $b\uparrow \infty$.
\end{proof}

We can now identify the law of $Z_a$.

\medskip
\noindent{\it Proof of Proposition \ref{process-exit} (i)}. We start from formula \eqref{LaplaceEM} and use first Lemma \ref{Laplace1}
and then Lemma \ref{Laplace2} to obtain, for every $\lambda\geq 0$,
\begin{align*}
E\Big[ \exp(-\lambda Z_a)]
&=E\Big[ \exp \Big(6\int_{L_a}^\infty \rd t\,\Big(\frac{1}{(R_t)^2}-\frac{1}{(R_t-(a-
(\frac{2\lambda}{3}+a^{-2})^{-1/2}))^2}\Big)\Big)\Big]\\
&= \Big(\frac{a-(a-
(\frac{2\lambda}{3}+a^{-2})^{-1/2})}{a}\Big)^3,
\end{align*}
which yields the desired result.
\hfill$\square$

\medskip
Our next goal is to obtain the law of the whole process $(Z_a)_{a\geq 0}$, where by
convention we take $Z_0=0$. To this end it is convenient to introduce a ``backward'' filtration
$(\g_a)_{a\geq 0}$, which we will define after introducing some notation. 
If $\mathrm{w}\in \mathcal{W}$, we set $\tau_a(\mathrm{w}):=\inf\{t\geq 0: \mathrm{w}(t)\notin(a,\infty)\}$, with
the usual convention $\inf\varnothing=\infty$. Then, let
$a\geq 0$ and $x>a$, and let $\omega=(\omega_s)_{s\geq 0}\in C(\R_+,\mathcal{W}_x)$ be such that
$\omega_s=x$ for all $s$ large enough. For every $s\geq 0$, we define ${\rm tr}_a(\omega)_s
\in \mathcal{W}_x$ by the formula
$${\rm tr}_a(\omega)_s= \omega_{\eta^{(a)}_s(\omega)},$$
where, for every $s\geq 0$,
$$\eta^{(a)}_s(\omega):=\inf\{r\geq 0:\int_0^r \rd u\,\ind{\{\zeta_{(\omega_u)}\leq \tau_a(\omega_u)\}} >s\}.$$
From the properties of the Brownian snake, it is easy to verify that $\N_x(\rd \omega)$ a.e.,
${\rm tr}_a(\omega)$ belongs to $C(\R_+,\mathcal{W}_x)$, and the paths ${\rm tr}_a(\omega)_s$
do not visit $(-\infty,a)$, and may visit $a$ only at their endpoint (what we have done is removing those paths that hit $a$ and
survive for some positive time after hitting $a$). Note that we are using a particular instance of the time change $\eta^D_s$ introduced when defining the
$\sigma$-field $\mathcal{E}^D$ in subsection \ref{prelisnake} (indeed, the $\sigma$-field
$\mathcal{E}^{(a,\infty)}$  is generated by the mapping $\omega\mapsto {\rm tr}_a(\omega)$ up
to negligible sets).

Recall formula \eqref{Poisson} for the point measure $\n$. For every $a\geq 0$, we let
$\g_a$ be the $\sigma$-field generated by the process $(R_{L_a+t})_{t\geq 0}$ and by
the point measure
$$\n^{(a)}:= \sum_{i\in I\cup J, {t_i}>L_a} \delta_{(t_i,{\rm tr}_a(\omega^i))}.$$
In the definition of $\n^{(a)}$, we keep only those excursions that start from the ``spine''
at a time greater than $L_a$ (so that obviously their initial point is greater than $a$) and we 
truncate these excursions at level $a$. Note that $\n^{(0)}=\n$. 

From our definitions it is clear that $\g_a\supset \g_b$ if $a<b$. Furthermore, it follows from
the measurability property of exit measures that $Z_a$ is $\g_a$-measurable, for
every $a>0$ (the point is that $\z_a(\omega^i)$ is a measurable function of 
${\rm tr}_a(\omega^i)$). We also notice that, for every $a>0$, the process $(R_{L_a+t})_{t\geq 0}$
is independent of $(R_t)_{0\leq t\leq L_a}$. This follows from last exit decompositions for 
diffusion processes, or in a more straightforward way this can be deduced from the 
time-reversal property already mentioned above. 

\begin{proposition}
\label{condi-Markov}
Let  $0<a<b$. Then, for every $\lambda\geq 0$,
$$E[\exp(-\lambda Z_a)\mid\g_b]
=\Big(\frac{b}{a+(b-a)(1+\frac{2\lambda a^2}{3})^{1/2}}\Big)^3
\;\exp\Big(-\frac{3Z_b}{2} \Big(\frac{1}{(b-a+(\frac{2\lambda}{3}+a^{-2})^{-1/2})^2}-\frac{1}{b^2}\Big)\Big).
$$
\end{proposition}

If $b>0$ is fixed, the proposition shows that that the process $(Z_{b-a})_{0\leq a<b}$ is 
time-inhomoge\-neous Markov with respect to the (forward) filtration $(\g_{b-a})_{0\leq a<b}$, and
identifies the Laplace transform of the associated transition kernels. Since the law
of $Z_b$ is also given by Proposition \ref{process-exit} (i), this completely characterizes the law
of the process $(Z_a)_{a\geq 0}$. The more explicit description of this law
given in Proposition \ref{process-exit} (ii) will be derived later. 

\begin{proof}
Recall that $0<a<b$ are fixed. We write
$$Z_a=Y_{a,b} + \wt Y_{a,b},$$
where
$$Y_{a,b}:=\sum_{i\in I\cup J, t_i> L_b}\z_a(\omega^i)
\;,\quad \wt Y_{a,b} :=\sum_{i\in I\cup J, L_a<t_i\leq L_b} \z_a(\omega^i)
.$$
From the fact that $(R_{L_b+t})_{t\geq 0}$
is independent of $(R_t)_{0\leq t\leq L_b}$ and properties of Poisson measures, it
easily follows that 
$Y_{a,b}$ and $\wt Y_{a,b}$ are independent, and more precisely $\wt Y_{a,b}$ is independent
of $\sigma(Y_{a,b})\vee \g_b$. This implies that
\begin{equation}
\label{Marko-tech1}
E[\exp(-\lambda Z_a)\mid\g_b]=E[\exp(-\lambda \wt Y_{a,b})]\,E[\exp(-\lambda Y_{a,b})\mid\g_b].
\end{equation}

From the special Markov property (see also the remark following Proposition \ref{SMP}), we have
\begin{align*}
E[\exp(-\lambda Y_{a,b})\mid\g_b]&
=E\Bigg[\prod_{i\in I\cup J, t_i> L_b} \exp(-\lambda \z_a(\omega^i))\,\Bigg|\, \g_b\Bigg]\\
&=\exp\Bigg(-\sum_{i\in I\cup J, t_i> L_b}\z_b(\omega^i)\N_b\Big(\ind{\{\r\subset(0,\infty)\}}(1-e^{-\lambda\z_a})\Big)
\Bigg)\\
&=\exp\Big(-Z_b\,\N_b\Big(\ind{\{\r\subset(0,\infty)\}}(1-e^{-\lambda\z_a})\Big)
\Big)\\
&=\exp\Bigg(-\frac{3Z_b}{2}
\Bigg(\Big( b -a +(\frac{2\lambda}{3}+a^{-2})^{-1/2} \Big)^{-2} -b^{-2}\Bigg)\Bigg),
\end{align*}
where the last equality is Lemma \ref{Laplace1}. 

Using Proposition \ref{process-exit} (i), we have thus,
\begin{align*}
E[\exp(-\lambda Y_{a,b})]
&=\Bigg(1+ b^2\Bigg(\Big( b -a +(\frac{2\lambda}{3}+a^{-2})^{-1/2} \Big)^{-2} -b^{-2}\Bigg)
\Bigg)^{-3/2}\\
&= \Big(\frac{b}{b -a +(\frac{2\lambda}{3}+a^{-2})^{-1/2}}\Big)^{-3},
\end{align*}
and since 
$Y_{a,b}$ and $\wt Y_{a,b}$ are independent,
\begin{align*}
E[\exp(-\lambda \wt Y_{a,b})]&= E[\exp(-\lambda Z_a)]\times (E[\exp(-\lambda Y_{a,b})])^{-1}\\
&=\Big(1+\frac{2\lambda a^2}{3}\Big)^{-3/2}\Big(\frac{b}{b -a +(\frac{2\lambda}{3}+a^{-2})^{-1/2}}\Big)^{3}\\
&= \Bigg(\frac{b}{a+(b-a)\Big(1+\frac{2\lambda a^2}{3}\Big)^{1/2}}\Bigg)^3.
\end{align*}
The statement of the proposition follows from \eqref{Marko-tech1} and the preceding calculations.
\end{proof}

We will now identify the transition kernels whose Laplace transform appears in the previous 
proposition. To this end, we recall the discussion of subsection \ref{CSBPpreli}, 
which we will apply with the particular value $c= \sqrt{8/3}$.

\begin{proposition}
\label{conditionedCSBP}
Let $\rho>0$ and $x>0$. The finite-dimensional marginal distributions of $(Z_{\rho-a})_{0\leq a\leq \rho}$
knowing that $Z_\rho=x$ coincide with those of the continuous-state branching process with
branching mechanism $\psi(u)=\sqrt{8/3}\,u^{3/2}$ started from $x$ and conditioned
on extinction at time $\rho$. 
\end{proposition}

\begin{proof}
Recall the notation introduced in subsection \ref{CSBPpreli}.
By comparing the right-hand side of \eqref{CSBPtech2} with the formula of
Proposition \ref{condi-Markov}, we immediately see that, for $0\leq s<t<\rho$,
$$E[\exp(-\lambda Z_{\rho-t})\mid\g_{\rho-s}]= \int e^{-\lambda y}\,\pi_{s,t}(Z_{\rho-s},dy).$$
Arguing inductively, we obtain that, for every $0<s_1<\ldots<s_p<\rho$, the conditional distribution 
of $(Z_{\rho-{s_1}},\ldots,Z_{\rho-s_p})$ knowing $\g_\rho$ is $\pi_{0,s_1}(Z_{\rho},\rd y_1)\pi_{s_1,s_2}(y_1,\rd y_2)\ldots
\pi_{s_{p-1},s_p}(y_{p-1},\rd y_p)$. The desired result follows.
\end{proof}

We can now complete the proof of Proposition \ref{process-exit}.

\medskip
\noindent{\it Proof of Proposition \ref{process-exit} (ii)}. 
We first verify that $Z_a$ and $\wt X_{-a}$ have the same distribution, for every fixed $a>0$. 
Let $\lambda >0$ and set $f(y)=e^{-\lambda y}$ to simplify notation. By the properties
of the process $\wt X$, we have
$$E[f(\wt X_{-a})]=\lim_{x\uparrow \infty} E_x[f(X_{T-a})\,\ind{\{a\leq T\}}],$$
where $T=\inf\{t\geq 0: X_t=0\}$ as previously.
On the other hand, recalling the definition of the functions $\phi_t$ in subsection \ref{CSBPpreli},
\begin{align*}
E_x[f(X_{T-a})\,\ind{\{a\leq T\}}]&=\lim_{n\uparrow\infty} \sum_{k=1}^\infty 
E_x[\ind{\{a+\frac{k-1}{n}<T\leq a+\frac{k}{n}\}}\, f(X_{k/n})]\\
&=\lim_{n\uparrow\infty} \sum_{k=1}^\infty 
E_x\Big[f(X_{k/n})\,P_{X_{k/n}}\Big(a-\frac{1}{n}<T\leq a\Big)\Big]\\
&=\lim_{n\uparrow\infty} \sum_{k=1}^\infty 
E_x\Big[f(X_{k/n})\int_{a-1/n}^a \phi_b(X_{k/n})\,\rd b\Big]\\
&=E_x\Big[\int_0^\infty f(X_t)\,\phi_a(X_t)\,\rd t\Big],
\end{align*}
where dominated convergence is easily justified by the fact that $E_x[T]<\infty$ and $\phi_b(0)=0$
for every $b>0$.
Now use the form of $\phi_a$
together with formula \eqref{CSBPtech} (with $c= \sqrt{8/3}$) to see that the right-hand side of the
last display is equal to
\begin{align*}
&\int_0^\infty \rd t\,\frac{3x}{a^3}\,(\lambda + \frac{3}{2a^2})^{-3/2}\,\Big((\lambda + \frac{3}{2a^2})^{-1/2}
+\sqrt{\frac{2}{3}} t\Big)^{-3}\,\exp\Big(-x\Big((\lambda + \frac{3}{2a^2})^{-1/2}
+\sqrt{\frac{2}{3}} t\Big)^{-2}\Big)\\
&\quad= (1+\frac{2\lambda a^2}{3})^{-3/2}\;\Big(1- \exp(-x(\lambda + \frac{3}{2a^2})^{-1/2})\Big).
\end{align*}
We then let $x\uparrow \infty$ to get that
$$E[\exp(-\lambda \wt X_{-a})]=(1+\frac{2\lambda a^2}{3})^{-3/2} = E[\exp(-\lambda Z_a)]$$
by assertion (i) of the proposition.  

Knowing that $Z_a$ and $\wt X_{-a}$ have the same distribution, the proof is completed as 
follows. We observe that, for every $a>0$, the law of $(\wt X_{-a+t})_{0\leq t\leq a}$ conditionally
on $\wt X_{-a}=x$ coincides with the law of $X$ started from $x$ and conditioned on
extinction at time $a$ (we leave the easy verification to the reader). By comparing with
Proposition \ref{conditionedCSBP}, we get the desired statement.
\hfill$\square$

\medskip
As a consequence of Proposition \ref{process-exit}, the process $(Z_r)_{r>0}$ has a c\`adl\`ag modification, and
from now on we deal only with this modification. We conclude this subsection by proving Proposition \ref{approx-exit}:
We need to verify that our definition of the random variable $Z_r$ matches the approximation given in 
this proposition. 

\medskip
\noindent{\it Proof of Proposition \ref{approx-exit}}.
If $x\in \t_\infty$ and $x$ is not on the spine, the point $\Pi(x)$ belongs to 
$B_r^\bullet(\mathcal{P}_\infty)^c\cap B_{r+\ve}(\mathcal{P}_\infty)$ 
if and only if $\Lambda_x\in(r,r+\ve]$ and $\Lambda_y>r$ for every $y\in \llbracket x,\infty\llbracket$.
Recalling our notation $\mathbf{V}$ for the volume measure on $\t_\infty$, we can thus write
$$|B_r^\bullet(\mathcal{P}_\infty)^c\cap B_{r+\ve}(\mathcal{P}_\infty)|
= \sum_{i\in I\cup J: t_i>L_r}
\mathbf{V}(\{x\in \t^i: \Lambda_x\leq r+\ve\hbox{ and } \Lambda_y>r,\;\forall y\in \llbracket \rho_i,x\rrbracket
\}).$$
We will first deal with indices $i$ such that 
$t_i>L_{r+\ve}$, and we set
$$A_\ve:=\sum_{i\in I\cup J: t_i>L_{r+\ve}}
\mathbf{V}(\{x\in \t^i: \Lambda_x\leq r+\ve\hbox{ and } \Lambda_y>r,\;\forall y\in \llbracket \rho_i,x\rrbracket
\})$$
to simplify notation. Recall that if $x\in \t^i$ and $x=p_{\zeta^i}(s)$, we have $\Lambda_x=\wh\omega^i_s$
and $\{\Lambda_y:y\in \llbracket \rho_i,x\rrbracket\}=\{\omega^i_s(t):0\leq t\leq \zeta^i_s\}$. 
An application of the special
Markov property shows that the conditional distribution of 
$A_\ve$
knowing $Z_{r+\ve}$ is the law of $U_\ve(Z_{r+\ve})$, where $U_\ve$ is a subordinator whose 
L\'evy measure is the ``law'' of
$$\int_0^\sigma \rd s\,\mathbf{1}_{\{\wh W_s \leq r+\ve;\; W_s(t)>r,\;\forall t\in[0,\zeta_s]\}},$$
under $\N_{r+\ve}$ (and $U_\ve$ is assumed to be independent of $Z_{r+\ve}$). From the first
moment formula for the Brownian snake \cite[Proposition IV.2]{LGZ}, one easily derives that
$$\N_{r+\ve}\Big(\int_0^\sigma \rd s\,
\mathbf{1}_{\{\wh W_s \leq r+\ve;\; W_s(t)>r,\;\forall t\in[0,\zeta_s]\}}\Big)=
E_{r+\ve}\Big[\int_0^\infty \rd t\,\mathbf{1}_{\{B_t\leq r+\ve\}}\,\mathbf{1}_{\{t<\gamma_r\}}\Big] = \ve^2,$$
where we have used the notation of the proof of Lemma \ref{Laplace2}. 
On the other hand, scaling arguments show that
$$(U_\ve(t))_{t\geq 0}\build{=}_{}^{\rm(d)} (\ve^4\, U_1(\frac{t}{\ve^2}))_{r\geq 0},$$
and the law of large numbers implies that $t^{-1}U_1(t)$ converges a.s. to $1$ as $t\to\infty$. 
Since the conditional distribution of $\ve^{-2} A_\ve$ knowing $Z_{r+\ve}$
is the law of $\ve^{2}U_1(\frac{Z_{r+\ve}}{\ve^2})$, it follows from the preceding observations
that 
$$\ve^{-2} A_\ve - Z_{r+\ve} \build{\la}_{\ve \to 0}^{} 0$$
in probability. Since $Z_{r+\ve}$ converges to $Z_r$ as $\ve\to 0$, we conclude that
$$\ve^{-2} A_\ve  \build{\la}_{\ve \to 0}^{} Z_r$$
in probability. To complete the proof, we just have to check that
$$\ve^{-2} \sum_{i\in I\cup J: L_r<t_i\leq L_{r+\ve}}
\mathbf{V}(\{x\in \t^i: \Lambda_x\leq r+\ve\hbox{ and } \Lambda_y>r,\;\forall y\in \llbracket \rho_i,x\rrbracket
\}) \build{\la}_{\ve \to 0}^{} 0$$
in probability. We leave the easy verification to the reader. 
\hfill$\square$

\subsection{The law of the volume of the hull}
\label{volu-hull}

This subsection is devoted to the proof of Theorem \ref{laws-hull}.
We recall our notation 
$B^\bullet_a(\pp_\infty)$ for the hull of radius $a$ in the Brownian plane $\pp_\infty$.
To simplify notation we write $B^\bullet_a$ instead of $B^\bullet_a(\pp_\infty)$, and we also write
$|B^\bullet_a|$ for the volume of this hull. Recall that
$Z_a$ is interpreted as a generalized length of the boundary of $B^\bullet_a$.

Thanks to the construction of the Brownian plane explained in
subsection \ref{newrep} and to formula \eqref{formula-hull}, we can express the volume $|B^\bullet_a|$ as the sum of two
independent contributions:
\begin{enumerate}
\item[$\bullet$] The total volume of those subtrees that branch off the spine
below level $L_a$.
\item[$\bullet$] The contribution of the subtrees that branch off the spine
above level $L_a$. More precisely, we need to sum, over all indices
$i\in I\cup J$ with $t_i>L_a$, the Lebesgue measure of the set of all $s\in[0,\sigma_i]$ such that
the path $\omega^i_s$ hits level $a$. Via an application of the special Markov property, the
conditional distribution of this contribution given $Z_a$ will follow from formula \eqref{Delma}
below.
\end{enumerate}

The beginning of this subsection is devoted to calculating the Laplace transform of
the first of these two contributions. Thanks to Theorem \ref{decotheo}, this 
is also the Laplace transform of $\sigma$ under the conditional
probability measure $\N_a(\cdot\mid W_*=0)$. This motivates the
following calculations. 

We recall the notation 
$\z_0$ for the (total mass of the) exit measure from $(0,\infty)$, and we also set
$$\y_0:=\int_0^\sigma \rd s\,\ind{\{\tau_0(W_s)=\infty\}},$$
where we recall that $\tau_0(\w)=\inf\{t\geq 0:\w(t)\notin(0,\infty)\}$. 
Our first goal is to compute, for every $\lambda,\mu> 0$, the fonction 
$u_{\lambda,\mu}(x)$ defined for every $x>0$ by
$$u_{\lambda,\mu}(x)= \N_x(1-\exp(-\lambda\z_0 -\mu \y_0)).$$
Note that $u_{\lambda,0}(x)$ is given by formula \eqref{Laplaceexit}. On the other
hand, the limit of $u_{\lambda,\mu}$ as $\lambda\uparrow\infty$ is
\begin{equation}
\label{Delma}
u_{\infty,\mu}(x):= \N_x(1-\ind{\{\r\subset(0,\infty)\}}\exp( -\mu \y_0))
= \sqrt{\frac{\mu}{2}}\Big(3 \coth^2((2\mu)^{1/4}x) -2\Big)
\end{equation}
by \cite[Lemma 7]{Del}. The latter formula is generalized in the next lemma. 

\begin{lemma}
\label{joint-Laplace}
We have, for every $x>0$:
\begin{enumerate}
\item[$\bullet$] if $\lambda > \sqrt{\frac{\mu}{2}}$,
$$u_{\lambda,\mu}(x)=
\sqrt{\frac{\mu}{2}}\Bigg( 3\Bigg(\coth\Bigg((2\mu)^{1/4} x +\coth^{-1}\sqrt{\frac{2}{3} +\frac{1}{3}
\sqrt{\frac{2}{\mu}}\lambda}\Bigg)\Bigg)^2 -2 \Bigg);$$
\item[$\bullet$] if $\lambda < \sqrt{\frac{\mu}{2}}$,
$$u_{\lambda,\mu}(x)=
\sqrt{\frac{\mu}{2}}\Bigg( 3\Bigg(\tanh\Bigg((2\mu)^{1/4} x +\tanh^{-1}\sqrt{\frac{2}{3} +\frac{1}{3}
\sqrt{\frac{2}{\mu}}\lambda}\Bigg)\Bigg)^2 -2 \Bigg).$$
\end{enumerate}
\end{lemma}

\rem If $\lambda = \sqrt{\frac{\mu}{2}}$, we have simply
$$u_{\lambda,\mu}(x)=\sqrt{\frac{\mu}{2}}.$$
This can be obtained by a passage to the limit from the previous formulas, but a direct
proof is also easy. 

\begin{proof}
By results due to Dynkin, the function $u_{\lambda,\mu}$ solves the differential equation
\begin{equation}
\label{equadiff}
\left\{\begin{array}{ll}
\frac{1}{2} u'' = 2u^2 -\mu\;,\quad&\hbox{on }(0,\infty),\\
\noalign{\smallskip}
u(0)=\lambda.&
\end{array}
\right.
\end{equation}
This is indeed a very special case of Theorem 3.1 in \cite{Dyn}. For the reader
who is unfamiliar with the general theory of superprocesses, a direct proof 
can be given along the lines of the proof of Lemma 6 in \cite{Del}. 

It is also easy to verify that
$$\lim_{x\to\infty} u_{\lambda,\mu}(x)= \N_0(1-e^{-\mu \sigma}) = \sqrt{\frac{\mu}{2}}.$$
The formulas of the lemma then follow by solving equation \eqref{equadiff}, which requires some tedious
but straightforward calculations.
\end{proof}

For future reference, we note that, if $\lambda > \sqrt{\frac{\mu}{2}}$, we have, for every $x>0$,
\begin{equation}
\label{flow1}
u_{\lambda,\mu}(x)=u_{\infty,\mu}(x+\theta_\mu(\lambda)),
\end{equation}
where the function $\theta_\mu$, which is defined on $( \sqrt{\frac{\mu}{2}},\infty)$ by
$$\theta_\mu(\lambda) = (2\mu)^{-1/4}\coth^{-1}\sqrt{\frac{2}{3} +\frac{1}{3}
\sqrt{\frac{2}{\mu}}\lambda},$$
is the functional inverse of $u_{\infty,\mu}$. Of course \eqref{flow1} is nothing but the flow property
of solutions of \eqref{equadiff}.

\begin{proposition}
\label{timebelow0}
Let $a>0$. Then, for every $\mu>0$,
$$\N_a\Big(e^{-\mu \y_0}\,\Big|\, W_*=0\Big)=
-\frac{1}{3}\,a^3\,u'_{\infty,\mu}(a)= a^3(2\mu)^{3/4} \,\frac{\cosh((2\mu)^{1/4}a)}{\sinh^3((2\mu)^{1/4}a)}.
$$
\end{proposition}

\rem The conditioning on $\{W_*=0\}$ may be understood as a limit as $\ve \to 0$ of
conditioning on $\{-\ve<W_*\leq 0\}$. Equivalently, we may use Theorem \ref{decotheo},
which  provides an explicit description of the conditional probabilities $\N_0(\cdot\mid W_*=y)$
for every $y<0$. We also note that under the conditioning $\{W_*=0\}$ we have
$\y_0=\sigma$. 

\begin{proof} We first observe that, for every $\ve>0$,
\begin{equation}
\label{below1}
\N_a(-\ve<W_*\leq 0)=\frac{3}{2a^2}-\frac{3}{2(a+\ve)^2} \build{\sim}_{\ve \to0}^{} \frac{3\ve}{a^3},
\end{equation}
by \eqref{hittingzero}. On the other hand,
\begin{align*}\N_a\Big(e^{-\mu \y_0}\,\ind{\{-\ve<W_*\leq 0\}}\Big)
&=\N_a\Big(e^{-\mu \y_0}\,\ind{\{\z_0>0,W_*>-\ve\}}\Big)\\
&=\N_a\Big(e^{-\mu \y_0}\,\ind{\{\z_0>0\}}\,\exp(-\z_0\,\N_0(W_*\leq -\ve))\Big)\\
&=\N_a\Big(\exp\Big(-\mu\y_0 - \frac{3}{2\ve^2}\z_0\Big)\;\ind{\{\z_0>0\}}\Big)
\end{align*}
using the special Markov property in the second equality, and then \eqref{hittingzero}. Set 
$\alpha=\frac{3}{2\ve^2}$ to simplify notation. Then,
\begin{align*}
\N_a\Big(\exp\Big(-\mu\y_0 - \alpha \z_0\Big)\ind{\{\z_0>0\}}\Big)
&=\N_a\Big(1-e^{-\mu\y_0}\ind{\{\z_0=0\}}\Big)-\N_a\Big(1-e^{-\mu\y_0-\alpha\z_0}\Big)\\
&=u_{\infty,\mu}(a)-u_{\alpha,\mu}(a)\\
&=u_{\infty,\mu}(a)-u_{\infty,\mu}(a+\theta_\mu(\alpha))\\
&\build{\sim}_{\alpha\to\infty}^{} -\theta_\mu(\alpha)\,u'_{\infty,\mu}(a),
\end{align*}
using \eqref{flow1} in the last equality. Since
$$\theta_\mu(\alpha)\build{\sim}_{\alpha\to \infty}^{} \sqrt{\frac{3}{2\alpha}},$$
it follows from the preceding discussion that
$$\N_a\Big(e^{-\mu \y_0}\,\ind{\{-\ve<W_*\leq 0\}}\Big)\build{\sim}_{\ve\to 0}^{} - u'_{\infty,\mu}(a)\,\ve.$$
The result of the proposition follows using also \eqref{below1}. 
\end{proof}

We state the next result in terms of super-Brownian motion, although our main motivation comes from
our application to the Brownian plane in Theorem \ref{laws-hull}. Recall that, in order to use the connection
with the Brownian snake, we always assume that the branching mechanism of super-Brownian
motion is $\psi_0(u)=2u^2$.

\begin{corollary}
\label{supercoro}
Let $a>0$ and $r>0$. Assume that $(\x_t)_{t\geq 0}$ is a super-Brownian motion that
starts from $r\delta_a$ under the probability measure $\P_{r\delta_a}$. Set
$$\Sigma=\int_0^\infty \rd t\,\langle \x_t,1 \rangle,$$
and write $\r^{\x}$ for the range of $\x$. Then, for every $\mu>0$,
\begin{align*}
&\E_{r\delta_a}\Big[e^{-\mu \Sigma}\,\Big|\,\min\r^\x=0\Big]\\
&\qquad= a^3(2\mu)^{3/4} \,\frac{\cosh((2\mu)^{1/4}a)}{\sinh^3((2\mu)^{1/4}a)}
\,\exp\Big(-r\Big(\sqrt{\frac{\mu}{2}}\Big(3 \coth^2((2\mu)^{1/4}a) -2\Big)-\frac{3}{2a^2}\Big)\Big).
\end{align*}
\end{corollary}

\begin{proof}
We may assume that  $(\x_t)_{t\geq 0}$ is constructed from
a Poisson point measure $\n$ with intensity $r\N_a$ via formula \eqref{SBMsnake}. Then,
we immediately verify that
$$\Sigma = \int \n(\rd \omega)\,\sigma(\omega)$$
and properties of Poisson measures lead to the formula
$$\E_{r\delta_a}\Big[e^{-\mu \Sigma}\,\Big|\,\min\r^\x=0\Big]
=\N_a\Big(e^{-\mu \sigma}\Big|\,\min\r=0\Big)\,\exp\Big(-r\N_a\Big((1-e^{-\mu\sigma})\ind{\{\min\r>0\}}\Big)\Big).$$
The first term in the right-hand side is given by Lemma \ref{timebelow0}. As for the second term
we observe that
$$\N_a\Big((1-e^{-\mu\sigma})\ind{\{\min\r>0\}}\Big)
= \N_a\Big(1-e^{-\mu\sigma}\ind{\{\min\r>0\}}\Big)-\N_a(\min\r\leq 0) = u_{\infty,\mu}(a)-\frac{3}{2a^2}.$$
This completes the proof.
 \end{proof}

\noindent{\it Proof of Theorem \ref{laws-hull}}. The first formula of the theorem is a straightforward consequence of the second one since
we know the distribution of $Z_a$. More precisely, using Proposition \ref{process-exit} (ii), we observe that
\begin{align*}
&E\Big[\exp\Big(-Z_a\Big(\sqrt{\frac{\mu}{2}}\Big(3 \coth^2((2\mu)^{1/4}a) -2\Big)-
\frac{3}{2a^2}\Big)\Big)\Big]\\
&\qquad = \Big(1 + \frac{2a^2}{3} \Big(\sqrt{\frac{\mu}{2}}\Big(3 \coth^2((2\mu)^{1/4}a) -2\Big)-
\frac{3}{2a^2}\Big)\Big)^{-3/2}\\
&\qquad = 3^{3/2}a^{-3} (2\mu)^{-3/4}\,\Big(3 \coth^2((2\mu)^{1/4}a) -2\Big)^{-3/2}.
\end{align*}
If we multiply this quantity by
$$a^3(2\mu)^{3/4} \,\frac{\cosh((2\mu)^{1/4}a)}{\sinh^3((2\mu)^{1/4}a)}$$
we get the desired formula for $E[\exp(-\mu|B^\bullet_a|)]$.

Not suprisingly, the second formula of Theorem \ref{laws-hull} is a consequence of the analogous formula 
in Corollary \ref{supercoro}. Let us explain this. Using our representation of the 
Brownian plane, and formula \eqref{formula-hull},
we can write $|B^\bullet_a|$ as the sum of two independent contributions:
\begin{enumerate}
\item[$\bullet$] The contribution of subtrees branching off the spine at a level
smaller than $L_a$. Using Theorem
\ref{decotheo}, we see
that this contribution is distributed as $\sigma$ under the conditional
probability measure $\N_a(\cdot\mid W_*=0)$. We also note that this contribution
is independent of the $\sigma$-field $\g_a$.
\item[$\bullet$] The contribution of subtrees branching off the spine at a level
greater than $L_a$. This contribution
is $\g_a$-measurable. Furthermore, an application of
the special Markov property (similar to the one in the proof of Proposition \ref{condi-Markov})
shows that its conditional distribution given $Z_a=r$ is the law of
$$\sum_{k\in K} \sigma(\omega_{(k)})$$
where $\sum_{k\in K} \delta_{\omega_{(k)}}$ is a Poisson measure with intensity $r\N_a(\cdot\cap\{W_*>0\})$.
\end{enumerate}

\noindent The preceding discussion shows that the conditional distribution of $|B^\bullet_a|$ given
$Z_a=r$ coincides with the distribution of $\Sigma$ under $\P_{r\delta_a}(\cdot\mid\min\mathcal{R}^X=0)$,
with the notation of Corollary \ref{supercoro}. This completes the proof.
\hfill$\square$

\subsection{The process of hull volumes}
\label{pro-hull}

Our goal in this subsection is to prove Theorem \ref{hull-process-description}.
In a way similar to Corollary \ref{supercoro}, we consider a super-Brownian
motion $(\x_t)_{t\geq 0}$, and the probability mesure $\P_{r\delta_0}$
under which this super-Brownian motion starts from $r\delta_0$. We also introduce the
associated historical process $(\mathbf{X}_t)_{t\geq 0}$. As previously, we may and will assume that $(\x_t)_{t\geq 0}$ and $(\mathbf{X}_t)_{t\geq 0}$ are
constructed from a Poisson measure
$$\n=\sum_{k\in K} \delta_{\omega_{(k)}}$$
with intensity $r\N_0$, via formulas \eqref{SBMsnake} and \eqref{historicalsnake}. We then set, for every $a<0$,
$$\mathscr{Z}_{a}=\sum_{k\in K} \z_a(\omega_{(k)})$$
and, for every $a\leq 0$,
$$\mathscr{Y}_{a}=\sum_{k\in K} \y_a(\omega_{(k)})$$
where
$$\y_a(\omega):=\int_0^\sigma \rd s\,\ind{\{\tau_a(W_s(\omega))=\infty\}}.$$
We also set 
$\mathscr{Z}_{0}=r$ by convention.

In the theory of superprocesses \cite{Dyn}, 
$\mathscr{Z}_{a}$ corresponds to the total mass of the exit measure of 
the historical process $(\mathbf{X}_t)_{t\geq 0}$
from $(a,\infty)$ (for our present purposes, we do not need this interpretation).  
We also note that, for every $a\leq 0$, we have
$$\mathscr{Y}_{a}=\int_0^\infty \rd t\,\int \mathbf{X}_t(\rd \w)\,\mathbf{1}_{\{\tau_a(\w)=\infty\}},$$
and the right-hand side is the total integrated mass of those historical paths that do not hit $a$.

As previously, $X=(X_t)_{t\geq 0}$ denotes a continuous-state branching process
with branching mechanism $\psi(u)=\sqrt{8/3}\,u^{3/2}$ that starts from $r$
under the probability measure $P_r$. We will use the ``L\'evy-Khintchine representation'' for $\psi$:
we have
$$\psi(u)=\int \kappa(dy)\,(e^{-\lambda y}-1+\lambda y)$$
where $\kappa(dy)$ is the measure on $(0,\infty)$ given by
$$\kappa(dy)= \sqrt{\frac{3}{2\pi}}\,y^{-5/2}\,dy.$$

\begin{proposition}
\label{jointfixedtime}
Let $a>0$. The law under $\P_{r\delta_0}$ of the pair $(\mathscr{Z}_{-a},\mathscr{Y}_{-a})$ 
coincides with the law under $P_r$ of the pair
\begin{equation}
\label{jointjump}
\Big(X_a,\sum_{i:s_i\leq a} \xi_i\,(\Delta X_{s_i})^2\Big)
\end{equation}
where $s_1,s_2,\ldots$ is a measurable enumeration of the jumps of $X$, and 
$\xi_1,\xi_2,\ldots$ is a sequence of i.i.d. real random variables with density
$$\frac{1}{\sqrt{2\pi x^5}}\,e^{-1/2x}\,\ind{(0,\infty)}(x),$$
which is independent of the process $(X_t)_{t\geq 0}$.
\end{proposition}

\begin{proof} We first observe that, for every $\lambda,\mu>0$, we have
$$\E_{r\delta_0}[\exp(-\lambda \mathscr{Z}_{-a}-\mu \mathscr{Y}_{-a})]=\exp(-r u_{\lambda,\mu}(a))$$
by the exponential formula for Poisson measures. We will prove that the
joint Laplace transform of the pair $\eqref{jointjump}$ is given by
the same expression.

To this end, we fix $\mu>0$ and write $\alpha=\sqrt{2\mu}$ to simplify
notation. We also set $w_a(\lambda)=u_{\lambda,\mu}(a)$ for every $a\geq 0$.
As a consequence of \eqref{equadiff} (or directly from Lemma \ref{joint-Laplace})
we have for every $a,b\geq 0$,
$$w_{a+b}=w_a\circ w_b$$
and $w_0(\lambda)=\lambda$. Furthermore, the derivative of $w_a(\lambda)$ at 
$a=0$ is easily computed from the formulas of Lemma \ref{joint-Laplace}:
\begin{equation}
\label{derivee1}
\frac{\rd}{\rd a}w_a(\lambda)_{|a=0}= \sqrt{\frac{2}{3}}\,\sqrt{\alpha +\lambda}\,(\alpha-2\lambda)
\end{equation}
where we recall that $\alpha=\sqrt{2\mu}$. 

Let us consider now the Laplace transform of the pair $\eqref{jointjump}$. We first observe
that the Laplace transform of the variables $\xi_i$ is given by
$$E[e^{-\beta \xi}]= (1+\sqrt{2\beta})\,e^{-\sqrt{2\beta}},$$
for every $\beta\geq 0$ (note that $E[\xi e^{-\beta \xi}]= e^{-\sqrt{2\beta}}$ by the well-known formula for
the Laplace transform of a positive stable random variable with parameter $1/2$). It follows that, for
every $\lambda >0$,
$$
E_r\Big[\exp\Big(-\lambda X_a- \mu \sum_{i:\xi_i\leq a} \xi_i(\Delta X_{s_i})^2\Big)\Big]
= E_r\Big[\exp(-\lambda X_a)\;\prod_{0\leq s\leq a} (1+\alpha \Delta X_s)e^{-\alpha \Delta X_s}\Big].
$$
The additivity property of continuous-state branching processes allows us to write
the right-hand side in the form $\exp(-rv_a(\lambda))$, where the function $v_a(\lambda)$
(which of course depends also on $\alpha$) is such that $v_0(\lambda)=\lambda$. The Markov
property of $X$ readily gives the semigroup property
$$v_{a+b}=v_a\circ v_b$$
for every $a,b\geq 0$. To complete the proof of the proposition, it suffices to verify that 
$w_a=v_a$ for every $a\geq 0$, and to this end it will be enough to prove that
\begin{equation}
\label{claimjoint}
\frac{\rd}{\rd a}w_a(\lambda)_{|a=0}=\frac{\rd}{\rd a}v_a(\lambda)_{|a=0}.
\end{equation}
The left-hand side is given by \eqref{derivee1}. Let us compute the right-hand side. We fix
$\lambda >0$ in what follows. 

As we already mentioned, the process $X$ is a Feller process with values in $[0,\infty)$. The
exponential function $\varphi_\lambda(x)= e^{-\lambda x}$ belongs to the domain of the generator
$\mathcal{L}$ of $X$, and
$$\mathcal{L}\varphi_\lambda(x)= \psi(\lambda)\,x\,\varphi_\lambda(x),$$
as a straightforward consequence of the formula for the Laplace transform of $X_t$. Consequently,
we have
$$e^{-\lambda X_t}= e^{-\lambda r} +M_t + \psi(\lambda)\int_0^t X_s\,e^{-\lambda X_s}\,\rd s,$$
where $M$ is a martingale, which is clearly bounded on every compact interval. For every $t\geq 0$, set
$$V_t:=\prod_{0\leq s\leq t} (1+\alpha \Delta X_s)e^{-\alpha \Delta X_s},$$
and note that $V$ is a nonnegative nonincreasing process, which is bounded by one. By applying 
the integration by parts formula, we have
\begin{equation}
\label{integparts}
V_te^{-\lambda X_t}= e^{-\lambda r}+ \int_0^t V_{s-}\,\rd M_s + \psi(\lambda)\int_0^t V_sX_s\,e^{-\lambda X_s}\rd s + \int_0^t e^{-\lambda X_s}\,\rd V_s.
\end{equation}
The martingale term $\int_0^t V_{s-}\,dM_s$ has zero expectation. Let us evaluate the expected value
of the last term
$$ \int_0^t e^{-\lambda X_s}\,\rd V_s=\sum_{0\leq s\leq t} e^{-\lambda X_s}\Delta V_s
=\sum_{0\leq s\leq t} e^{-\lambda X_{s-}}V_{s-} \times e^{-\lambda \Delta X_s}\Big(
(1+\alpha \Delta X_s)e^{-\alpha \Delta X_s} -1\Big).$$
We note that the dual predictable projection of the random measure 
$$\sum_{s\geq 0,\Delta X_s>0} \delta_{(s,\Delta X_s)}(\rd u,\rd x)$$
is the measure
$$X_u\,\rd u\,\kappa(\rd x)$$
where we recall that $\kappa(dx)$ is the ``L\'evy measure'' associated with $X$ (a simple way 
to get this is to use the Lamperti transformation to represent $X$ as a time-change of the
L\'evy process with L\'evy measure $\kappa$). It follows that
$$E\Big[ \int_0^t e^{-\lambda X_s}\,\rd V_s\Big]
= E\Big[\int_0^t X_sV_s e^{-\lambda X_s}\,\rd s\Big] \times \int \kappa(\rd x)\,e^{-\lambda x}
\Big((1+\alpha x)e^{-\alpha x} -1\Big).$$
By taking expectations in \eqref{integparts}, we thus get
$$e^{-rv_t(\lambda)}-e^{-rv_0(\lambda)}
=E\Big[\int_0^t X_sV_s e^{-\lambda X_s}\,\rd s\Big] \times\Big( \psi(\lambda)+ \int \kappa(\rd x)\,e^{-\lambda x}
\Big((1+\alpha x)e^{-\alpha x} -1\Big)\Big).$$
Note that 
$$\frac{1}{t} E\Big[\int_0^t X_sV_s e^{-\lambda X_s}\,\rd s\Big]  
\build{\la}_{t\downarrow 0}^{} r\,e^{-\lambda r},$$
and thus it immediately follows from the preceding display that
\begin{align*}
\frac{\rd}{\rd a}v_a(\lambda)_{|a=0}&= - \psi(\lambda)-\int \kappa(\rd x)\,e^{-\lambda x}
\Big((1+\alpha x)e^{-\alpha x} -1\Big)\\
&= - \int \kappa(\rd x)\,\Big((1+\alpha x)e^{-(\alpha +\lambda)x} -1 +\lambda x\Big).
\end{align*}
From the expression of $\kappa$, straightforward calculations lead to the formula
$$\int \kappa(\rd x)\,\Big((1+\alpha x)e^{-(\alpha +\lambda)x} -1 +\lambda x\Big)= -\sqrt{\frac{2}{3}}\,\sqrt{\alpha +\lambda}\,(\alpha-2\lambda)$$
and our claim \eqref{claimjoint} follows, recalling \eqref{derivee1}. This completes the proof.
\end{proof}

With the notation introduced in Proposition \ref{jointfixedtime}, set 
for every $a\geq 0$,
$$Y_a:=\sum_{i:s_i< a} \xi_i\,(\Delta X_{s_i})^2.$$

\begin{corollary}
\label{lawprocess}
The law of the process $(\mathscr{Z}_{-a},\mathscr{Y}_{-a})_{a\geq 0}$ under
$\P_{r\delta_0}$ coincides with the law of $(X_a,Y_a)_{a\geq 0}$ under $P_r$.
\end{corollary}

\begin{proof} An application of the special Markov property shows that the 
process $(\mathscr{Z}_{-a},\mathscr{Y}_{-a})_{a\geq 0}$ is (time-homogeneous)
Markov under $\P_{r\delta_0}$, with transition kernel given by
$$\E_{r\delta_0}[ g(\mathscr{Z}_{-a-b},\mathscr{Y}_{-a-b}) \mid (\mathscr{Z}_{-a},\mathscr{Y}_{-a})]
=\Phi_b(\mathscr{Z}_{-a},\mathscr{Y}_{-a}),$$
where
$$\Phi_b(z,y)=\E_{z\delta_0}[g(\mathscr{Z}_{-b},y+\mathscr{Y}_{-b})].$$
On the other hand, the Markov property of the continuous-state branching process
$X$ also shows that the process $(X_a,Y_a)_{a\geq 0}$ is Markov under $P_r$, and
$$E_r[g(X_{a+b},Y_{a+b})\mid (X_a,Y_a)] = \Psi_b(X_a,Y_a),$$
where 
$$\Psi_b(z,y)=E_z[g(X_b,y+Y_b)].$$
By Proposition \ref{jointfixedtime}, we have $\Phi_b=\Psi_b$, and the desired result follows.
\end{proof}

\noindent{\bf Remark.} We chose to put a strict inequality $s_i<a$ in the definition of
$Y_a$ so that the process $Y$ has left-continuous paths, which is also the case for $\mathscr{Y}_{-a}$.
On the other hand, both $\mathscr{Z}_{-a}$ and $X_a$ have right-continuous paths. 

\medskip
\noindent{\it Proof of Theorem \ref{hull-process-description}}. Fix $\rho>0$, and let $U$ follow a Gamma distribution 
with parameter $\frac{3}{2}$ and mean $\rho^2$, so that $U$ has the same distribution as $Z_\rho$,
by Proposition
\ref{process-exit} (i). Suppose that, conditionally given $U$, $\mathcal{N}$ is a Poisson point measure with
intensity $U\,\N_0$ under the probability measure $\P$. We can use formulas \eqref{SBMsnake} and \eqref{historicalsnake} to define a 
super-Brownian motion $(\x_t)_{t\geq 0}$ started from $U\,\delta_0$ and the associated historical superprocess. We then define
 $(\mathscr{Z}_{a},\mathscr{Y}_{a})_{a\leq 0}$
as in the beginning of this subsection. We also write $S$ for the
extinction time of $\x$.

The arguments used
in the proof of Theorem \ref{laws-hull}, based on our representation
of the Brownian plane and formula \eqref{formula-hull}, show that the process
$(Z_{\rho-a},|B^\bullet_\rho|-|B^\bullet_{\rho-a}|)_{0\leq a\leq \rho}$ has the same
distribution as $(\mathscr{Z}_{-a},\mathscr{Y}_{-a})_{0\leq a\leq \rho}$ under
$\P(\cdot \mid S=\rho)$. For a precise justification, note that $B^\bullet_\rho\backslash B^\bullet_{\rho-a}$
is the image under $\Pi$ of those $x\in \t_\infty$ such that $\Lambda_y> \rho-a$
for every $y\in\llbracket x,\infty\llbracket$ and there exists  $z\in\llbracket x,\infty\llbracket$
such that $\Lambda_z\leq \rho$. If $x$ satisfies these properties, either $x$ belongs to
one of the subtrees branching off the spine at a level belonging to $] L_{\rho-a},L_\rho
]$, or $x$ belongs to
one of the subtrees branching off the spine at a level greater than $L_\rho$ but the label of 
one of the ancestors of $x$ in this subtree is less than or equal to $\rho$ (and, in both cases,
the labels of the ancestors of $x$ in the subtree containing $x$ remain strictly 
greater than $\rho-a$). The volume of the set of points $x$ corresponding to the
second case is handled via the  special Markov property for the domain $(\rho,\infty)$,
in a way similar to the end of the proof of Theorem \ref{laws-hull}.
We obtain that
the sum of the two contributions leads to the quantity $\mathscr{Y}_{-a}$ for a super-Brownian 
motion starting from $Z_\rho\delta_0$ and conditioned on extinction at time $\rho$. 

Write $P_{(U)}$ for a probability measure under which the continuous-state branching process
$X$ starts from $U$ (and the process $Y$ is constructed by the formula preceding
Corollary \ref{lawprocess}), and let $T$ be the extinction time of $X$ as previously. Also set
for every $a\geq 0$,
$$\wt Y_a=\sum_{i:\tilde s_i\geq -a} \xi_i\,(\Delta \wt X_{\tilde s_i})^2,$$
where $\wt s_1,\wt s_2,\ldots$ is a measurable enumeration of the jumps of $\wt X$, and the 
random variables $\xi_i$ are as in Proposition \ref{jointfixedtime} and are supposed to be
independent of $\wt X$. 

From
Corollary \ref{lawprocess}, we obtain that the law of $(\mathscr{Z}_{-a},\mathscr{Y}_{-a})_{0\leq a\leq \rho}$ under
$\P(\cdot \mid S=\rho)$ coincides with the law of $(X_a,Y_a)_{0\leq a\leq \rho}$ under 
$P_{(U)}(\cdot \mid T=\rho)$. However, using the final observation of the
proof of Proposition \ref{process-exit} (ii), the latter law is also the law of 
$(\wt X_{-\rho+a}, \wt Y_\rho - \wt Y_{\rho-a})_{0\leq a\leq \rho}$.

Summarizing, we have obtained the identity in distribution
$$(Z_{\rho-a},|B^\bullet_\rho|-|B^\bullet_{\rho-a}|)_{0\leq a\leq \rho}
\build{=}_{}^{\rm(d)} (\wt X_{-\rho+a}, \wt Y_\rho - \wt Y_{\rho-a})_{0\leq a\leq \rho}.$$
This immediately gives
$$(Z_{a},|B^\bullet_{a}|)_{0\leq a\leq \rho}
\build{=}_{}^{\rm(d)} (\wt X_{-a}, \wt Y_{a})_{0\leq a\leq \rho},$$
from which the statement of Theorem \ref{hull-process-description} follows. 
\hfill$\square$

\section{Asymptotics for the UIPQ}
\label{asympUIPQ}

We will rely on the Chassaing-Durhuus construction of the UIPQ \cite{CD}. The fact that
this construction is equivalent to the more usual construction involving local limits 
of finite quadrangulations can be found in \cite{Men}. The Chassaing-Durhuus construction
is based on a random infinite labeled discrete ordered tree, which we denote here by
$\T$. In a way very analogous to the tree $\t_\infty$ considered above, the tree $\T$ consists
of a spine, which is a discrete half-line, and for every vertex of the spine, of two finite subtrees
grafted at this vertex respectively to the left and to the right of the spine (if the grafted subtree
consists only of the root, this means that we add nothing). The root of $\T$
is the first vertex of the spine. The set of all corners of $\T$ is equipped with a total order 
induced by the clockwise contour exploration of the tree.
Each vertex $v$ of $\T$ is assigned a positive integer label $\ell_v$, 
in such a way that the label of the root is $1$ and the labels of two neighboring vertices may
differ by at most  $1$ in absolute value. We will not need the exact distribution of the
tree $\T$: See e.g. \cite[Section 2.3]{LGM}.

Let us now explain the construction of the UIPQ from the tree $\T$. First the vertex
set of $Q_\infty$ is the union of the vertex set $V(\T)$ of $\T$ and of an extra vertex denoted by $\partial$. 
We then generate the edges of $Q_\infty$ by the following device, which is analogous to the
Schaeffer bijection between finite (rooted) quadrangulations and well-labeled trees \cite{CS}. All corners 
of $\T$ with label $1$ are linked to $\partial$ by an edge of $Q_\infty$. Any other corner $c$
is linked by an edge of $Q_\infty$ to the last corner before $c$ (in the clockwise countour order) with strictly smaller label.
The resulting collection of edges forms an infinite quadrangulation of the plane, 
which is the UIPQ $Q_\infty$ (see Fig.~2, and \cite{CD} for more details). It easily follows from the construction that the graph distance
(in $Q_\infty$) between $\partial$ and another vertex of $Q_\infty$
is just the label of this vertex in $\T$. 

\begin{figure}[!h]
\label{d-hull}
 \begin{center}
 \includegraphics[width=8cm]{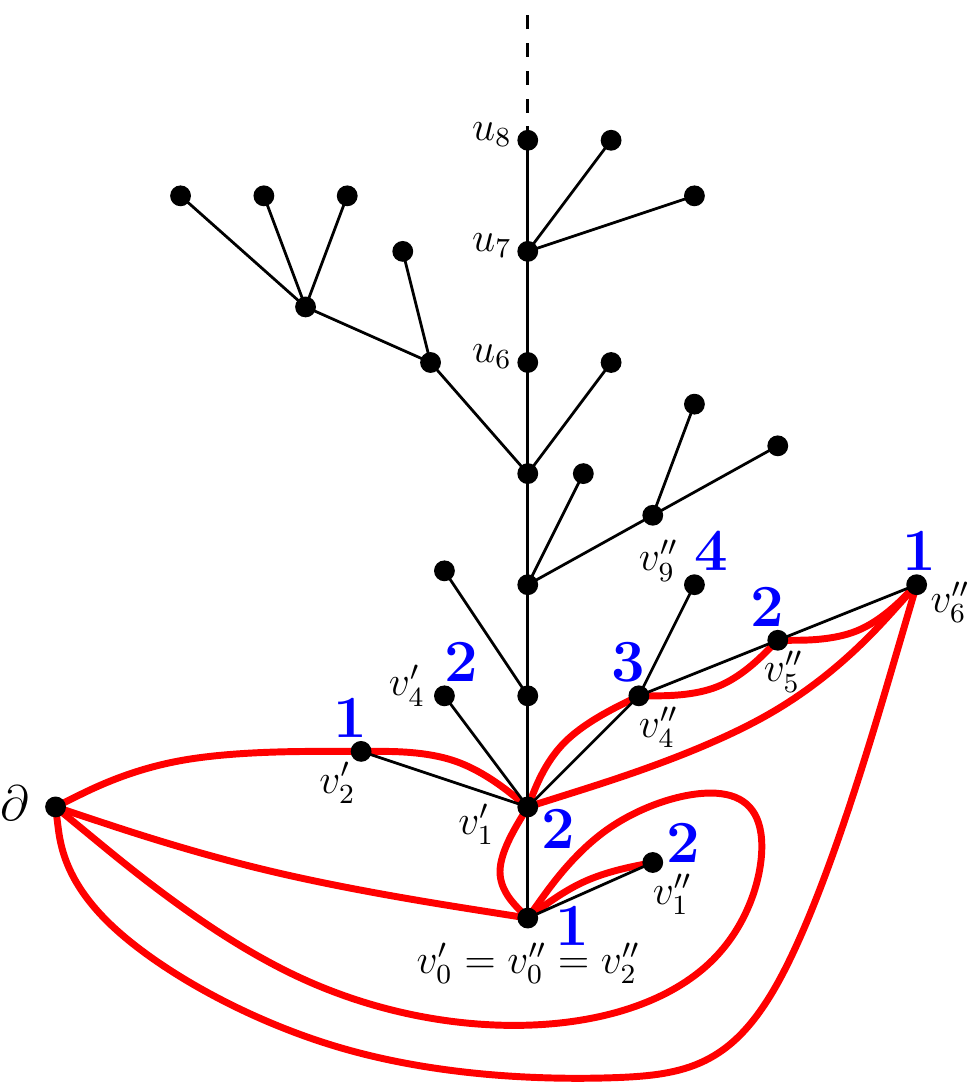}
 \caption{The Chassaing-Durhuus construction. The tree $\T$ is represented in thin lines. A few of the vertices
 $v'_k,v''_k$ have been indicated together with their labels in bold figures. The edges
 of $Q_\infty$ incident to $4$ faces have been drawn in thick lines.
 }
 \end{center}
 \vspace{-2mm}
 \end{figure}

Let us introduce the left and right contour processes. Starting from the
root of $\T$, we list all corners of the left side of $\T$ in clockwise contour order, and, for 
every $k\geq 0$, we denote the vertex corresponding to the $k$-th corner in this enumeration by $v'_k$ 
(in such a way that $v'_0$ is the root of $\T$). We then write $C^{(L)}_k$ for the generation (distance from the root in $\T$) 
of $v_k$, and $V^{(L)}_k= \ell_{v'_k}$. Note that $|C^{(L)}_{k+1}-C^{(L)}_k|=1$
for every $k\geq 0$. We define similarly 
$C^{(R)}_k$ and $V^{(R)}_k$ using the exploration in counterclockwise order of the right
side of the tree, and the analog of the sequence $(v'_k)_{k\geq 0}$ is denoted by $(v''_k)_{k\geq 0}$. By linear interpolation, we may view all four processes 
$C^{(L)},V^{(L)},C^{(R)},V^{(R)}$ as indexed by $\R_+$. A key ingredient of the 
following proof is the convergence  \cite[Theorem 5]{LGM},
\begin{equation}
\label{keyconv}
\Big(\frac{1}{k^2}C^{(L)}_{k^4s},\sqrt{\frac{3}{2}} \,\frac{1}{k}\,V^{(L)}_{k^4s},
\frac{1}{k^2}C^{(R)}_{k^4s},\sqrt{\frac{3}{2}} \,\frac{1}{k}\, V^{(R)}_{k^4s}\Big)_{s\geq 0}
\build{\la}_{k\to\infty}^{\rm (d)} \Big( h(\Theta'_s),\Lambda_{\Theta'_s},
h(\Theta''_s),\Lambda_{\Theta''_s}\Big)_{s\geq 0},
\end{equation}
where we recall that $\Theta'_s$ and $\Theta''_s$ are the exploration processes of respectively the
left and the right side of $\t_\infty$ (see subsection \ref{newrep}), and we use the notation
$h(\Theta'_s)=d_\infty(0,\Theta'_s)$ for the ``height'' of $\Theta'_s$ in $\t_\infty$. The convergence
\eqref{keyconv} holds in the sense of weak convergence of the laws on the space $C(\R_+,\R^4)$. 
We also mention another convergence in distribution concerning labels on the spine.
Write $u_n$ for the $n$-th vertex on the spine of $\T$. Then,
\begin{equation}
\label{labelspine}
\Big(\sqrt{\frac{3}{2}} \,\frac{1}{k}\,\ell_{u_{\lfloor k^2 s\rfloor}}\Big)_{s\geq 0}
\build{\la}_{k\to\infty}^{\rm (d)}
(R_s)_{s\geq 0}
\end{equation}
and this convergence in distribution holds jointly with \eqref{keyconv}. The convergence
\eqref{labelspine} can be found in \cite[Proposition 1]{LGM}. The fact that this convergence 
holds jointly with \eqref{keyconv} is clear from the proof of Theorem 5 in \cite{LGM}. 

According to \cite[Lemma 3]{LGM}, we have for every $A>0$,
\begin{equation}
\label{transient-discrete}
\lim_{K\to\infty}\Big(\sup_{k\geq 1} P\Big(\inf_{t\geq K} \frac{1}{k}\,V^{(L)}_{k^4t} < A\Big)\Big) =0,
\end{equation}
and by symmetry the analogous statement with $V^{(L)}$ replaced by $V^{(R)}$ also holds. 
Finally, we note that Lemma \ref{transience-lemma} implies
\begin{equation}
\label{transient-cts}
\lim_{s\uparrow \infty} \Lambda_{\Theta'_s} = \lim_{s\uparrow \infty} \Lambda_{\Theta''_s}=+\infty,
\quad\hbox{a.s.}
\end{equation}

\begin{figure}[!h]
\label{d-hull}
 \begin{center}
 \includegraphics[width=14cm]{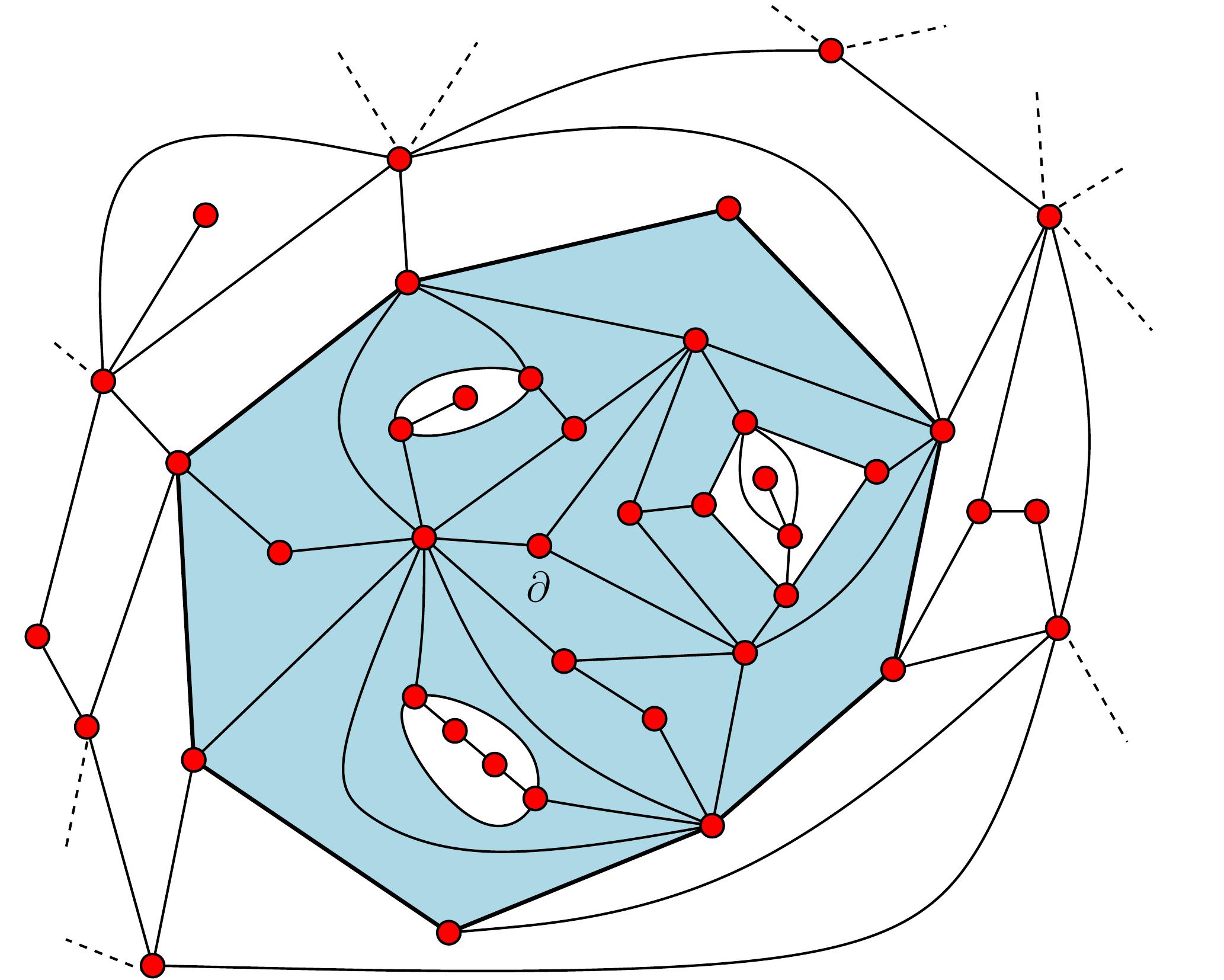}
 \caption{A representation of the UIPQ near the vertex $\partial$. The shaded part 
 is the ball $\b_2(Q_\infty)$. The hull  $\b^\bullet_2(Q_\infty)$, whose boundary is in
 thick lines on the figure, is obtained by filling in the holes of $\b_2(Q_\infty)$.}
 \end{center}
 \vspace{-2mm}
 \end{figure}

For every integer $k\geq 1$, define the ball $\b_k(Q_\infty)$ as the union of all
faces of $Q_\infty$ that are incident to (at least) one vertex at distance smaller than or 
equal to $k-1$ from $\partial$. The hull $\b^\bullet_k(Q_\infty)$ is  then obtained 
by adding to $\b_k(Q_\infty)$ the bounded components of the complement of 
$\b_k(Q_\infty)$ (see Fig.~3). Define the ``volume'' $|\b^\bullet_k(Q_\infty)|$ as the number of faces
contained in $\b^\bullet_k(Q_\infty)$.

\begin{theorem}
\label{hull-UIPQ}
We have
$$(k^{-4}\,|\b^\bullet_{\lfloor kr\rfloor}(Q_\infty)|)_{r>0} \build{\la}_{k\to\infty}^{\rm (d)} (\frac{1}{2}|B_{r\sqrt{3/2}}^\bullet(
\mathcal{P}_\infty)|)_{r>0},$$
in the sense of weak convergence of finite dimensional marginals.
\end{theorem}

\noindent {\bf Remarks.} (i) In the companion paper \cite{CLG2}, we use the peeling process to
give a different approach to the convergence of the sequence of processes $(k^{-4}\,|\b^\bullet_{\lfloor kr\rfloor}(Q_\infty)|)_{r>0}$.
The limit then appears in the form given in Theorem \ref{hull-process-description}.

\noindent (ii)  By scaling, the processes $(\frac{1}{2}|B_{r\sqrt{3/2}}^\bullet(
\mathcal{P}_\infty)|)_{r>0}$ and $(|B_{(9/8)^{1/4}r}^\bullet(
\mathcal{P}_\infty)|)_{r>0}$ have the same distribution, and we recover the ``usual''
constant $(9/8)^{1/4}$ (see e.g. \cite{CS}). The reason for stating the theorem in
the form above is the fact that the convergence then holds jointly with \eqref{keyconv}
or \eqref{labelspine}, as the proof will show. 

\begin{proof} 
Instead of dealing with $|\b^\bullet_{\lfloor kr\rfloor}(Q_\infty)|$ we will consider the
quantity $\|\b^\bullet_{\lfloor kr\rfloor}(Q_\infty)\|$ defined as the number of vertices that
are incident to a face of $\b^\bullet_{\lfloor kr\rfloor}(Q_\infty)$. It is an easy exercise to
verify that the desired convergence will follow if we can prove that the statement holds
when $|\b^\bullet_{\lfloor kr\rfloor}(Q_\infty)|$ is replaced by $\|\b^\bullet_{\lfloor kr\rfloor}(Q_\infty)\|$
(the underlying idea is the fact that a finite quadrangulation with $n$ faces has 
$n+2$ vertices, and we also observe that, for every fixed $r>0$, the size of the boundary of $\b^\bullet_{\lfloor kr\rfloor}(Q_\infty)$
is negligible with respect to $k^4$ -- this is clear if we know that the sequence $(k^{-4}\|\b^\bullet_{\lfloor kr\rfloor}(Q_\infty)\|)_{r>0}$
converges to a limit which is continuous in probability). 

We will verify that, if $r>0$ is fixed, the sequence
$k^{-4}\|\b^\bullet_{\lfloor kr\rfloor}(Q_\infty)\|$ converges in distribution to $\frac{1}{2}|B_{r\sqrt{3/2}}^\bullet(
\mathcal{P}_\infty)|$. It will be clear that our method extends to a joint
convergence in distribution if we consider a finite number
of values of $r$, yielding the desired statement. To simplify the presentation, we take $r=1$
in what follows.
So our goal is to show that
\begin{equation}
\label{claim-UIPQ}
k^{-4}\|\b^\bullet_{k}(Q_\infty)\|\build{\la}_{k\to\infty}^{\rm (d)} \frac{1}{2}|B_{\sqrt{3/2}}^\bullet(
\mathcal{P}_\infty)|.
\end{equation}

If $u\in V(\T)$, write ${\rm Geo}(u\to\infty)$ for the geodesic path from $u$ to $\infty$ in $\T$, and 
set 
$$m(u):=\min\{\ell_v: v\in {\rm Geo}(u\to\infty)\}.$$
Let $k\geq 1$. We note that:
\begin{enumerate}
\item[\rm(i)] The condition $m(u)\geq k +3$ ensures that $x\notin \b^\bullet_{k}(Q_\infty)$. Indeed,
from the way edges of $Q_\infty$ are generated, it is easy to construct a path of $Q_\infty$ from 
$u$ to $\infty$ that visits only vertices at distance (at least) $m(u)-1$ from $\partial$. If $m(u)-1\geq k+2$,
none of these vertices can be incident to a face of $\b_{k}(Q_\infty)$.
\item[\rm(ii)] If $m(u)\leq k$ then $x\in \b^\bullet_{k}(Q_\infty)$. This is an immediate 
consequence of the discrete ``cactus bound'' (see \cite[Proposition 4.3]{CLM},
in a slightly different setting), which implies that
any path of $Q_\infty$ going from $u$ to $\infty$ visits a vertex at distance less than or equal to $m(u)$ from $\partial$.
\end{enumerate}

Recall our definition of the ``contour sequence'' $(v'_k)_{k\geq 0}$ of the left side of the tree.
We now extend the definition of $v'_k$ to nonnegative real indices: If $k\geq 1$ and $k-1<s<k$, 
we take $v'_s=v'_k$ if $C^{(L)}_k=C^{(L)}_{k-1}+1$ and $v'_s=v'_{k-1}$ if $C^{(L)}_k=C^{(L)}_{k-1}-1$. This definition is
motivated by the fact that we have $\int_0^\infty ds\,\mathbf{1}\{v'_s=u\} = 2$ for 
every vertex $u$ in the left side of $\T$ (not on the spine), and the same integral
is equal to $1$ if $u$ is on the spine and different from the root. We extend similarly the
definition of $v''_k$. 
 
We next observe that, for every fixed $s>0$,
\begin{equation}
\label{keyargument}
\frac{1}{k}\,m(v'_{k^4 s})\build{\la}_{k\to\infty}^{\rm (d)}
\sqrt{\frac{2}{3}}\, \min\{\Lambda_y:y\in \llbracket \Theta'_s,\infty
\llbracket\},
\end{equation}
and this convergence holds jointly with \eqref{keyconv}. The convergence
\eqref{keyargument} is essentially a
consequence of \eqref{keyconv} and \eqref{labelspine}. Let us only sketch the argument. A first technical ingredient
is to replace $m(v'_{k^4 s})$ by a truncated version obtained by replacing
${\rm Geo}(u\to\infty)$ in the definition of $m(u)$ by the geodesic from $u$
to the vertex $u_{Ak^2}$, for some large integer constant $A$. One then proves, using 
\eqref{keyconv} and \eqref{labelspine}, that
the analog of \eqref{keyargument} holds for this truncated version, with a limit equal 
to $\sqrt{2/3}\, \min\{\Lambda_y:y\in \llbracket \Theta'_s,A
\rrbracket\}$ (a convenient way is to use a minor variant of the homeomorphism theorem of \cite{MM}
to see that \eqref{keyconv} implies also the convergence of the associated ``snakes'', which is
what we need here). Finally, the fact that the convergence of truncated versions 
suffices to get \eqref{keyargument} is easy using \eqref{transient-discrete} and
\eqref{transient-cts}. 

If we consider a finite number of values of $s$, the corresponding convergences \eqref{keyargument}
hold jointly (and jointly with \eqref{keyconv}). Via the method of moments, it easily
follows that, for every $A>0$, and every $a>0$,
$$\int_0^A \rd s\,\mathbf{1}_{\{m(v'_{k^4 s} )\leq a\,k\}}
\build{\la}_{k\to\infty}^{\rm (d)} \int_0^A \rd s\,\mathbf{1}_{\{\min\{\Lambda_y:y\in \llbracket \Theta'_s,\infty
\llbracket\}\leq \sqrt{3/2}\,a\}}.$$
Thanks to \eqref{transient-discrete} and
\eqref{transient-cts}, we can replace $A$ by $\infty$ and obtain
$$\int_0^\infty \rd s\,\mathbf{1}_{\{m(v'_{k^4 s} )\leq a\,k\}}
\build{\la}_{k\to\infty}^{\rm (d)} \int_0^\infty \rd s\,\mathbf{1}_{\{\min\{\Lambda_y:y\in \llbracket \Theta'_s,\infty
\llbracket\}\leq \sqrt{3/2}\,a\}}.$$
By combining this convergence with the analogous result for the right side of the tree, we get
\begin{align}
\label{techUIPT}
&\int_0^\infty \rd s\,\mathbf{1}_{\{m(v'_{k^4 s} )\leq a\,k\}}+ \int_0^\infty \rd s\,\mathbf{1}_{\{m(v''_{k^4 s} )\leq a\,k\}}\nonumber\\
&\quad\build{\la}_{k\to\infty}^{\rm (d)} \int_0^\infty \rd s\,\mathbf{1}_{\{\min\{\Lambda_y:y\in \llbracket \Theta'_s,\infty
\llbracket\}\leq \sqrt{3/2}\,a\}}
+ \int_0^\infty \rd s\,\mathbf{1}_{\{\min\{\Lambda_y:y\in \llbracket \Theta''_s,\infty
\llbracket\}\leq \sqrt{3/2}\,a\}}.
\end{align}
By \eqref{formula-hull}, the limit in the previous display is equal to
$ |B^\bullet_{\sqrt{3/2}\,a}(\pp_\infty)|$. On the other hand, previous remarks show that, if $a\,k\geq 1$,
$$\int_0^\infty \rd s\,\mathbf{1}_{\{m(v'_{k^4 s} )\leq a\,k\}}+ \int_0^\infty \rd s\,\mathbf{1}_{\{m(v''_{k^4 s} )\leq a\,k\}}= 2\,k^{-4}(\#\{u\in V(\T): m(u)\leq a\,k\} -1).$$
Furthermore, it follows from properties (i) and (ii) stated above that
$$\#\{u\in V(\T): m(u)\leq k\} \leq \| \mathcal{B}^\bullet_k(Q_\infty)\|\leq \#\{u\in V(\T): m(u)\leq k+2\}.$$
Our claim \eqref{claim-UIPQ} now follows from the convergence \eqref{techUIPT}
and the preceding observations, together with the fact that the mapping $r\mapsto |B^\bullet_{r}(\pp_\infty)|$
is continuous in probability. This completes the proof. 
\end{proof}

Let us conclude with a comment. It would seem more direct to derive Theorem \ref{hull-UIPQ}
from the fact that the Brownian plane is the Gromov-Hausdorff scaling limit of the UIPQ
\cite[Theorem 2]{CLG}. We refrained from doing so because the local Gromov-Hausdorff convergence
does not give enough information to handle volumes of balls or hulls. It would have been 
necessary to establish a type of Gromov-Hausdorff-Prokhorov convergence in our setting, in the spirit of 
the work of Greven, Pfaffelhuber and Winter \cite{GPW}, who however consider the case
of metric spaces equipped with a probability measure. This would require a number of
additional technicalities, and for this reason we preferred to rely on the results of \cite{LGM}.

\end{document}